\documentclass[12pt]{article}

\usepackage{here}
\usepackage{natbib}
\usepackage[dvipdfmx]{graphicx,color}
\usepackage{bm}
\usepackage{amsthm}
\usepackage{mathrsfs}
\usepackage{amsmath}
\usepackage{amssymb}
\usepackage{amsfonts}
\usepackage{multicol}
\usepackage{color}
\usepackage{indentfirst}
\usepackage{enumitem}
\usepackage{geometry}
\usepackage{xcolor}

\newcommand{\field}[1]{\mathbb{#1}}

\newcommand{\R}{\field{R}}

\newcommand{\Z}{\field{Z}}

\newcommand{\B}{\mathcal{B}}

\newcommand{\tr}{\mathrm{tr}}

\newcommand{\Int}{\mathrm{Int}}
\newcommand{\E}{\field{E}}

\makeatletter
	\newtheorem{theorem}{Theorem}[section]
	\newtheorem{lemma}{Lemma}[section]

\theoremstyle{definition}
	
	\newtheorem{remark}{Remark}[section]
	
	\newtheorem{assumption}{Assumption}[section]
\makeatother

\newcommand{\plim}{\xrightarrow{{\cal P}}}
\newcommand{\dlim}{\xrightarrow{{\cal L}}}

\makeatletter
 
  \@addtoreset{equation}{section}
\makeatother
\allowdisplaybreaks[4]

\begin{document}

\title{Change point detection in autoregressive models  with no moment assumptions }

\author{
{\small Fumiya Akashi} \\
{\small Waseda University} \\
{\small Department of Applied Mathematics} \\
{\small 169-8555, Tokyo, Japan} \\
\and
{\small Holger Dette } \\
{\small Ruhr-Universit\"at Bochum } \\
{\small Fakult\"at f\"ur Mathematik } \\
{\small 44780 Bochum, Germany } \\
\and
{\small Yan Liu} \\
{\small Waseda University} \\
{\small Department of Applied Mathematics} \\
{\small 169-8555, Tokyo, Japan} \\
}

\date{}

\maketitle
\begin{abstract}
In this paper we consider the problem of detecting a  change   in the parameters of an autoregressive process, where the moments of the innovation process do not necessarily exist.
An empirical likelihood ratio  test for the existence of a change point is proposed and its asymptotic properties are studied.
In contrast to other work
on change point tests using empirical likelihood,
we do not assume knowledge
of the location of the change point.
In particular, we prove
that the maximizer of the empirical likelihood is a consistent estimator for  the parameters of the autoregressive model in the case of no change point and
derive the limiting distribution of the corresponding test statistic under  the null hypothesis. We also establish consistency of the new test. A nice feature of the method consists in the fact  that the resulting test is asymptotically
distribution free and does not require an estimate of the long run variance.
The asymptotic properties of the test are investigated by means of a small simulation study, which demonstrates good finite sample properties of the proposed method.
 \end{abstract}

\parindent  0cm
\noindent


\section{Introduction}
The problem of detecting  structural breaks in time series   has been studied for a long time.
Since the seminal work of  \cite{page1954, page1955}, who proposed a sequential scheme
for identifying  changes in the mean of a sequence of independent random variables, numerous authors have worked on this problem. A large part of the literature concentrates
on CUSUM tests, which are nonparametric by design [see \cite{auehor2013} for a recent review and some  important references].
Other authors make distributional assumptions to construct tests for structural breaks. For example,  \cite{gombay1990}  suggested
 a likelihood ratio  procedure to test for a change in the mean and  extensions of this
  method can be found in  the monograph of  \cite{csorgo1997} and the reference therein.
  An important problem in this context is the detection of changes in the parameters of an autoregressive process
  and we refer to the work of \cite{andrews1993}, \cite{bai1993,bai1994}, \cite{davhuayao1995}, \cite{lee2003} and \cite{berkesetal2011}  among others who proposed CUSUM-type and  likelihood ratio tests.

In practice, however, the distribution of random variables is rarely known and its misspecification may result in an
invalid analysis using  likelihood ratio  methods.  One seminal method to treat the likelihood ratio empirically
has been  investigated by \cite{O:1988},  \cite{qin1994empirical} in a general context and extended by \cite{chuangchan2002}
to estimate and test parameters in an autoregressive model.  In change point analysis the empirical likelihood approach
can be viewed as  a compromise between the completely  parametric  likelihood ratio  and  nonparametric CUSUM
method. \cite{baragona2013empirical} used this concept to
construct a test for change-points and showed that in
  the case where the location of the break points is   known,
the limiting distribution of the corresponding test statistic is a chi-square distribution.
 \cite{ciupsall2014}  considered the change point problem
 in a non-linear model with independent data without assuming knowledge of its location
and derived an extreme value distribution as limit distribution of the empirical likelihood ratio test statistic.
These findings  are similar in spirit to the meanwhile classical  results  in  \cite{csorgo1997}, who considered the likelihood ratio test.

The purpose of the present  paper is to  investigate an empirical likelihood test for a  change in  the parameters of an autoregressive process
with infinite variance (more precisely we do not assume the existence of any moments).
Our work is motivated by the fact
that in many  fields, such as electrical engineering, hydrology, finance and physical systems, one often observes
 ``heavy-tailed'' data [see  \cite{nolan2015stable} or  \cite{samoradnitsky1994stable} among many others].
To deal  with such  data,  many authors  have  developed $L_1$-based methods.
For example, \cite{chen2008analysis} constructed  a robust test  for a linear hypothesis of the parameters
based on  least absolute deviation.  \cite{ling2005self} and
\cite{pan2007weighted}
proposed self-weighted least absolute deviation-based estimators for (parametric)   time series models with an infinite variance innovation process and
show the asymptotic normality of the estimators.  However, the limit distribution of the $L_1$-based statistics  usually contains the
unknown  probability density  of the innovation process, which is difficult to estimate.
For example, \cite{ling2005self} and \cite{pan2007weighted}
used kernel density estimators for this purpose, but the choice of  the corresponding  bandwidth is not clear and often depends on users.

To circumvent problems of this type in the context of change point analysis, we  combine in  this paper quantile regression  and empirical likelihood
methods. As a remarkable feature, the asymptotic distribution of the
proposed test statistic does not involve unknown quantities of the model even if we consider autoregressive models with an infinite variance
in the innovation process. We would also like to emphasize that the nonparametric CUSUM tests proposed by \cite{bai1993,bai1994} for detecting structural breaks in the parameters of an autoregressive process assume the existence of the variance of the innovations. However, an alternative to the method proposed here  are
CUSUM  tests  based on quantile regression,   which has been  recently considered  by  \cite{qu2008},  \cite{suxiao2008}   and  \cite{zhouwangtang2015}
among others.

The remaining part  of this paper is organized as follows.
In Section \ref{sec2}, we introduce the model, the testing problem
and the so-called self-weighted empirical likelihood ratio test statistic.
Our main results are given in Section \ref{sec3},  where we derive the
limit distribution of the proposed test statistic and prove consistency. The finite sample properties 
 of the proposed test  are investigated   in Section \ref{sec4} by means of a simulation study. We also
 compare  the test proposed in this paper with the CUSUM test using quantile regression [see \cite{qu2008}].
While the empirical likelihood based test suggested here
is competitive with the CUSUM test using quantile regression when the innovation process is Gaussian,
it  performs remarkably better than the   CUSUM test of \cite{qu2008}
if  the innovation process  has heavy tails.
Moreover, the new test is robust with respect  non-stationarity even when the process is nearly a unit root process.
Finally,
 rigorous proofs of  the results relegated to   Section \ref{sec5}.

\section{Change point tests using empirical likelihood } \label{sec2}

Throughout this paper the following  notations and symbols are  used.
The set of all integers and real numbers are denoted as $\Z$ and $\R$, respectively.
For any sequence of random vectors $\{A_n:n\geq1\}$ we denote by
$$A_n\plim A ~\mbox{  and } ~A_n\dlim A
$$
 convergence in probability and law to a random vector $A$, respectively.
The transpose of a matrix $M$ is denoted by $M'$, and
$\|M\| = \{ \tr(M' M) \} ^{1/2}$ is the Frobenius norm.
We denote the $i$-dimensional zero vector, the $j\times k$ zero matrix and the $l\times l$ identity matrix by
$0_i$, $O_{j\times k}$ and $I_{l\times l}$,
respectively.

Consider  the autoregressive model  of order $p$ (AR($p$) model)  defined by
\begin{align} \label{eq:model}
y_t = X_{t-1}'\beta + e_t,
\end{align}
where $X_{t-1} = (y_{t-1},\ldots,y_{t-p})'$ and $\beta \in \R^p$
and assume that the innovation process $\{e_t:t\in\Z\}$
is a sequence of independent and identically distributed
(i.i.d.) random variables  with  vanishing median.
Let $\{y_{1-p},\ldots,y_n\}$ be an observed stretch from the model \eqref{eq:model} for
$\beta = \beta_0$, where $\beta_0 = (\beta_1,\ldots,\beta_p)'$ denotes the ``true'' parameter.

This paper focuses on a posteriori type change point problem for the parameters in the  AR($p$) process \eqref{eq:model}. More precisely,
we consider  the model
\begin{align*}
y_t = \left\{
\begin{array}{ll}
X_{t-1}'\theta_1 + e_t & (1\leq t \leq k^*)\\
X_{t-1}'\theta_2 + e_t & (k^*+1\leq t \leq n)
\end{array}
\right.
\end{align*}
for some vector $\theta_1,\theta_2\in\R^p$, where $k^*\in\{1,\ldots,n\}$ is the  unknown time point
of the change. The testing problem for a change point in the autoregressive process can then be formulated by the following hypotheses:
\begin{align}\label{eq:01}
H_0: \theta_1 = \theta_2 = \beta_0 \quad  \text{against} \quad
H_1: \theta_1\neq \theta_2.
\end{align}
Note that we neither assume knowledge of the change point $k^*$ (if the null hypothesis is not true) nor
of  the true value $\beta_0\in\R^p$ (if the null hypothesis holds). \\

For the testing problem \eqref{eq:01}, we construct an empirical likelihood ratio (ELR) test. To be precise,
 let ${\mathbb I}$ denote the indicator function.
 As the median of $e_t$ is zero, the  moment condition
\begin{align}
\E\Big[\Big\{
\frac{1}{2}-{\mathbb I} ( y_t - X_{t-1}'\beta_0 \leq 0 )
\Big\}a^*(X_{t-1})
\Big] = 0_{m} \label{eq:03}
\end{align}
holds under the null hypothesis $H_0$ in \eqref{eq:01},
where $a^*(X_{t-1})$ is any $m$-dimensional measurable function of $X_{t-1}$
independent of $e_t$.
Motivated by the moment conditions \eqref{eq:03},
 we first introduce the self-weighted moment function
\begin{align}
g(\mathscr{Y}_{t}^{p} , \beta) := \Big\{
\frac{1}{2}-{\mathbb I}\left( y_t - X_{t-1}'\beta \leq 0 \right)
\Big\}
a^*(X_{t-1}) \quad(t=1,\ldots,n),\notag
\end{align}
where  $\mathscr{Y}_{t}^{p} = (y_t, \dots, y_{t - p})$ and
$a^*(X_{t-1}) = w_{t-1}a(X_{t-1})$,
$a(x)=(x',\varphi(x)')'$ is an $m=(p+q)$-dimensional function,
$\varphi$ a $q$-dimensional function,
$w_{t-1} = w(y_{t-1},\ldots,y_{t-p})$ a self-weight and $w$  some positive weight function.
We can choose the weight function $w$ and $\varphi$ arbitrarily
provided that Assumption \ref{ass:2} in Section \ref{sec3} holds. In particular,
we can use $a(x)=x$, which corresponds to the case $q=0$ (see also Section \ref{sec4}).

Note that under the null hypothesis $H_0$,  we have
that $\E[g(\mathscr{Y}_{t}^{p}, \beta_0)] = 0_m$
 for all $t=1,\ldots,n$.
Let $r_{n, k}$ be $(v_{1},\ldots,v_k,v_{k+1},\ldots,v_n)'$ be a vector in the unit cube $[0,1]^n$, then the
  empirical likelihood (EL), for $\beta = \theta_1$ before the change point $k \in \{ 1,\ldots,n \}$
and $\beta = \theta_2$ after the change point, is defined by
\begin{align*}
L_{n, k}(\theta_1, \theta_2) := \sup\Big\{ \Big(\prod_{i=1}^k v_i\Big)\Big(\prod_{j=k+1}^n  v_j\Big)
: r_{n, k} \in {\cal P}_{n, k}\cap{\cal M}_{n, k}(\theta_1,\theta_2)
\Big\},
\end{align*}
where ${\cal P}_{n, k}$ and
${\cal M}_{n, k}(\theta_1,\theta_2)$ are subsets of the cube   $[0,1]^{n}$ defined as
\[
{\cal P}_{n, k}
:= \Big\{r_{n, k}\in[0,1]^{n}:
\sum_{i=1}^k v_i = \sum_{j=k+1}^n v_j=1
\Big\}\notag\\
\]
and
\begin{align*}
{\cal M}_{n, k}(\theta_1,\theta_2)
:= \Big\{r_{n, k}\in[0,1]^{n}:\sum_{i=1}^k
v_ig(\mathscr{Y}_{i}^{p} , \theta_1)
 = \sum_{j=k+1}^n v_j g(\mathscr{Y}_{j}^{p}, \theta_2) = 0_m
\Big\}.
\end{align*}
Note that the unconstrained maximum EL is represented as
\begin{align*}
L_{n, k, E}:=\sup\Big\{\prod^n_{i=1} v_i: r_{n, k} \in {\cal P}_{n, k}
\Big\} = k^{-k}(n-k)^{-(n-k)},
\end{align*}
and hence, the logarithm of the  empirical likelihood ratio (ELR) statistic is given by
\begin{align}
l_{n, k}(\theta_1,\theta_2)
&:= -\log\frac{L_{n, k}(\theta_1, \theta_2)}{L_{n, k, E}}\notag\\
&=-\log\sup\Big\{ \Big(\prod_{i=1}^k k v_i\Big)\Big(\prod_{j=k+1}^n (n-k)  v_j\Big)
: r_{n, k} \in {\cal P}_{n, k}\cap{\cal M}_{n, k}(\theta_1, \theta_2)\Big\} \nonumber\\
& = \Big[
	\sum_{i=1}^k\log\big\{1 - \lambda'g(\mathscr{Y}_{i}^{p}, \theta_1) \big\}
	+\sum_{j=k+1}^n\log\big\{1 - \eta'g(\mathscr{Y}_{j}^{p}, \theta_2)\big\}
	\Big],
\label{eq:08}
\end{align}
where \eqref{eq:08} is obtained by the Lagrange multiplier method and the multipliers $\lambda$,
 $\eta\in\R^m$ satisfy
\begin{align*}
 \sum_{i=1}^k \frac{g(\mathscr{Y}_{i}^{p}, {\theta_1})}{1-\lambda'g(\mathscr{Y}_{i}^{p}, \theta_1)}=
 \sum_{j=k+1}^n \frac{g(\mathscr{Y}_{j}^{p}, {\theta_2})}{1-\eta'g(\mathscr{Y}_{j}^{p}, \theta_2)}
 = 0_{m}.
\end{align*}

We finally  define the test statistic for   the change point problem \eqref{eq:01}.
Since the maximum ELR  under $H_0$ is given by
\begin{equation*} 
P_{n, k} := \sup_{\beta\in {\cal B}}\{-l_{n, k}(\beta,\beta)\},
\end{equation*} 
one may define the ELR test statistic by
\begin{equation} \label{tn}
T_n := 2 \max_{\lfloor r_1 n \rfloor \leq k\leq \lfloor r_2 n \rfloor }P_{n, k},
\end{equation}
where $ 0 < r_1 < r_2 < 1$ for fixed constants. Note that we do not consider the  maximum of $\{ P_{n, k}~|~k=1, \ldots , n\}$ as
 $P_{n, k}$ can not be estimated accurately for small and large values of $k$ (see Theorem \ref{thm:1} in Section \ref{sec3}
 for more details).
The asymptotic properties of  a weighted version of  this statistic  are investigated in the following section.
\begin{remark}
The approach presented here can be  naturally extended to the general $\tau$-quantile regression
models. To be precise, suppose that
$$Q_y(\tau \mid X_{t-1}) = \inf\{y: P(y_t<y \mid X_{t-1})\geq\tau\}$$
denotes the $\tau$th-quantile
of $y_t$ conditional on $X_{t-1}$ and assume that
$Q_y(\tau \mid X_{t-1}) = \beta(\tau)'X_{t-1}$.
The moment condition
\[
\E[g^{(\tau)}(\mathscr{Y}^{p}_t, \beta_0(\tau))]=0_m
\]
still holds under the null hypothesis $H_0$, if we define
\begin{align*}
g^{(\tau)}(\mathscr{Y}^{p}_t, \beta(\tau)) := \psi_\tau(y_t-\beta(\tau)'X_{t-1})
a^*(X_{t-1})
\end{align*}
and
$\psi_\tau(u) := \left\{\tau-{\mathbb I}(u\leq 0)\right\}$.
\end{remark}

\begin{remark}
The method can also be extended to develop change point analysis based on the generalized
empirical likelihood (GEL).
A GEL test  statistic for the change point problem \eqref{eq:01} can be defined by
\[
l^\rho_{n, k}(\theta_1,\theta_2) = 2
\Big[
\sup_{\lambda\in \R^m}\sum_{i=1}^k\rho\left\{\lambda'g(\mathscr{Y}_{i}^{p}, \theta_1) \right\}
+\sup_{\eta\in\R^{m}}\sum_{j=k+1}^n\rho\left\{\eta'g(\mathscr{Y}_{j}^{p}, \theta_2) \right\}
\Big ],
\]
where $\rho$ is a real-valued, concave, twice differentiable function defined on an open interval of the real line that contains 
the point $0$ with 
$\rho'(0) = \rho''(0) = 1$.
Typical examples for the choice of $\rho$ are given by $\rho(\nu) = -\log(1-\nu)$ and
\begin{equation} \label{eq:2.13}
\rho(\nu) = \frac{(1 + c \nu)^{(c+1)/c} - 1}{c + 1}.
\end{equation}
Using Lagrangian multipliers, it is easy to see that the choice $\rho (\nu)=- \log (1- \nu)$ yields the  empirical likelihood method discussed so far.
The class associated with \eqref{eq:2.13} is called the Cressie-Read family [see \cite{cressie1984multinomial}].
\end{remark}

\section{Main results}
\label{sec3}
In this section we state our main results.
Throughout this paper, let $F$ and $f$ denote the distribution function and
the probability density function of $e_t$, respectively.
We impose the following assumptions.

\begin{assumption}\label{ass:1}\mbox{}
\begin{enumerate}[label=(\roman*)]
\item $\beta_0\in\Int({\cal B})$, where the parameter space ${\cal B}$ is a compact set in $\R^p$ with non-empty interior.
\item $1-\beta_{1}z - \cdots-\beta_{p}z^p\neq 0$ for $|z|\leq 1$ and $\beta\in{\cal B}$.
\item The median of $e_t$ is zero.
\item The distribution function $F$ of $e_t$ is continuous and differentiable at the point $0$ with positive derivative $F^\prime (0)=f(0)$.
\end{enumerate}
\end{assumption}
\begin{assumption}\label{ass:2}
$\E[(w_{t-1} + w_{t-1}^2)(\|a(X_{t-1})\|^2 + \|a(X_{t-1})\|^3)]<\infty$.
\end{assumption}

\begin{assumption}\label{ass:3}
The matrix $\E[g(\mathscr{Y}^{p}_t, \beta_0)g(\mathscr{Y}^{p}_t, \beta_0)']$ is positive definite.
\end{assumption}
\noindent

\begin{assumption}\label{ass:4}\mbox{}
\begin{enumerate}[label=(\roman*)]
\item There exists a constant $\gamma>2$ such that
$\E[\|a^*(X_{t-1})\|^\gamma]<\infty$.
\item
Let $v_t:=\textrm{sign}(e_t)a^*(X_{t-1})$. Then the sequence $\{v_t:t\in\Z\}$ is strong mixing
with mixing coefficients $\alpha_l$ that satisfy
$\sum^{\infty}_{l=1}\alpha_l^{1-2/\gamma}<\infty$.
\end{enumerate}
\end{assumption}

\medskip

The {\it  maximum EL estimator}  $\hat\beta_{n,k}$ is defined by
\[
-l_{n, k}(\hat\beta_{n,k},\hat\beta_{n,k}) = \sup_{\beta\in{\cal B}}\{-l_{n, k}(\beta,\beta)\} ~, 
\]
 and the consistency with corresponding rate of convergence of this statistic
are  given in the following theorem.
\begin{theorem}\label{thm:1}
Suppose that Assumptions \ref{ass:1}-\ref{ass:4} hold and define  $k^*: = r n$ for some $r \in (0, 1)$.
Then, under the null hypothesis $H_0$, we have, as $n\to\infty$,
\[
\hat\beta_{n,k^*} - \beta_0 = O_p\left(n^{-1/2}\right).
\]
\end{theorem}

As seen from Theorem \ref{thm:1}, $T_n$ is not accurate for small $k$ and $n - k$  as the result
does not hold if $k/n= o(1)$ or $(n-k)/n= o(1)$.
In addition, the ELR statistic is not computable for small $k$ and $n-k$.
For this reason, we consider in the following discussion
the trimmed and weighted-version of EL ratio test statistic, defined by
\begin{align}
\tilde T_n := 2\max_{k_{1n}\leq k\leq k_{2n}}
h\Bigl(\frac{k}{n}\Bigr)P_{n, k}, \label{stat}
\end{align}
where $h$ is a given weight function,
$k_{1n}: = r_1 n$, $k_{2n}: = r_2 n$
and $0 < r_1 < r_2 < 1$.
If $\tilde T_n$ takes a significant large value, we have enough reason to reject the null hypothesis $H_0$ of no change point.
We also need a further assumption to control a remainder terms
in the stochastic expansion of $\tilde T_n$.

\begin{assumption}\label{ass:5}
$\sup_{0<r<1} h(r)^2<\infty$.
\end{assumption}
\noindent
With this additional assumption the limit distribution of the test statistic \eqref{stat} can be derived in the following theorem.
\begin{theorem}\label{thm:2}
Suppose that Assumptions \ref{ass:1}-\ref{ass:5} hold.
Then, under the null hypothesis $H_0$ of no change point
\begin{equation}  \label{weak}
\tilde T_n \dlim
T :=
\sup_{r_1 \leq r \leq r_2}
\left\{r^{-1}(1-r)^{-1} h(r) \big\|B(r)- r  B(1)\big\|^2
+
h(r) B(1)'Q  B(1)\right\}
\end{equation}
as $n\to\infty$.
Here $\{B(r):r\in[0,1]\}$ is an $m$-dimensional vector of independent Brownian
motions and the matrix  $Q$ is defined by
\begin{equation}\label{qmat}
Q = I_{m\times m}-\Omega^{-1/2}G\Sigma G' \Omega^{-1/2},
\end{equation}
where $A^{1/2}$ denotes  the square root of a nonnegative definite matrix $A$,
$G = G(\beta_0) =  {\partial g(\beta_0)}/ {\partial\beta'}$,  $\Sigma=(G'\Omega^{-1}G)^{-1}$ and
\begin{equation} \label{0mat}
\Omega :=
\E[g(\mathscr{Y}^{p}_t, \beta_0)g(\mathscr{Y}^{p}_t, \beta_0)']
= \frac{1}{4}\E[a^*(X_{t-1})a^*(X_{t-1})'].
\end{equation}
\end{theorem}

A test for the hypotheses in \eqref{eq:01} is now easily obtained by rejecting the null hypothesis in \eqref{eq:01} whenever
\begin{equation}  \label{test}
\tilde T_n  > q_{1-\alpha},
\end{equation}
where  $q_{1-\alpha} $ is the $({1-\alpha})$-quantile of the distribution of the random variable $T$ defined on the right-hand side of equation
\eqref{weak} (using an appropriate estimate of the matrix $Q$).

\begin{theorem}\label{thm:3}
Suppose that Assumptions \ref{ass:1}-\ref{ass:5} and  the alternative $H_1: \theta_1\neq \theta_2$ hold.
Then we have
\begin{equation*}  
\tilde T_n \plim \infty
\end{equation*}
as $n\to\infty$.
\end{theorem}
\noindent
Theorem \ref{thm:3} shows that the power of the test \eqref{test} approaches $1$ at any fixed alternative.
In other words, the test is consistent.

\section{Finite sample properties}  \label{sec4}
In this section, we illustrate the finite sample properties of  the ELR  test \eqref{test} for the hypothesis \eqref{eq:01} by means of small simulation study.
For this purpose we  consider  the AR(1) model
\begin{align*}
y_t = \beta y_{t-1} + e_t,
\end{align*}
where the  coefficient $\beta$ satisfies
\begin{align*}
\beta = \left\{
\begin{array}{ll}
\theta_1 & (t=1,\ldots,k^*)\\
\theta_2 & (t=k^*+1,\ldots,n)
\end{array}
\right.~.
\end{align*}
For the calculation of the ELR  statistic $\tilde{T}_n$ in \eqref{stat},
we use the functions $a(x)=x$ and $h(r) = r(1- r)$ throughout this section.
Following  \cite{ling2005self}, the self-weights are chosen as
\begin{align*}
w_{t-1} = \left\{
\begin{array}{ll}
1 & (d_{t-1}=0)\\
(c/d_{t-1})^3 & (d_{t- 1}\neq0)
\end{array}
\right.,
\end{align*}
where $d_{t-1} = |y_{t-1}|\mathbb{I}(|y_{t-1}|>c)$
and $c$ is the $95\%$-quantile of the sample
$\{y_0,y_1,\ldots,y_n\}$. The  trimming parameters in the definition of the statistic $\tilde T_n$   are chosen as $r_{1n}=0.1$
and $r_{2n}=0.9$. The critical value in \eqref{test} is obtained  as the empirical $95\%$ quantile
of the Monte-Carlo samples
\[
\left\{
\max_{k_{1n}\leq k\leq k_{2n}}
\left(B^{(l)}(k/n) - (k/n)B^{(l)}(1)\right)^2: l=1,\ldots,1000
\right\},
\]
where $B^{(1)}(\cdot), \ldots,  B^{(1000)}(\cdot)$ are  independent standard Brownian motions
(note that  in this case,       the matrix in \eqref{qmat} is given by $Q=0$).

In Figures \ref{fig:1}-\ref{fig:3},
we display the rejection probabilities of the ELR test \eqref{test} for the hypothesis
\eqref{eq:01}, where the  nominal level is chosen as $\alpha = 0.05$.
The horizontal and vertical axes show, respectively,
the values of $\theta_2$ and the rejection rate of the hypothesis
$H_0:\theta_1 = \theta_2$ at this point ($\theta_1$ is fixed as $0.3$).
The sample sizes are given by $n=100,200$ and $400$ and the distribution of the innovation process is a standard normal
distribution (Figure \ref{fig:1}), a $t$-distribution with $2$ degrees of freedom (Figure \ref{fig:2}) and a  Cauchy distribution (Figure \ref{fig:3}).
We also consider two values of the parameter $r$ in the definition of the change point  $k^* = r n$,
that is $r=0.5$ and $r=0.8$.

We observe that for small sample sizes, the test is slightly conservative and that the approximation of the nominal level improves with increasing sample size. The alternatives are rejected with reasonable probabilities, where the power is larger in the case $r= 0.5$ than for $r= 0.8$. A comparison of the different distributions in Figures \ref{fig:1}-\ref{fig:3} shows that the power is lower for standard normal distributed innovations, while an error process with a Cauchy distribution yields the largest rejection probabilities. Other simulations show a similar picture, and the results are omitted for the sake of brevity.

\begin{figure}[H]
\centering
\caption{\it  Simulated rejection probabilities of the ELR test \eqref{test} in the AR(1) model
with  normal distributed innovations. } \label{fig:1}
\medskip
\medskip

\begin{tabular}{cc}
(a) $\theta_1 = 0.3$, $r=0.5$ & (b) $\theta_1 = 0.3$, $r=0.8$\\
  \includegraphics[width=0.5\hsize]{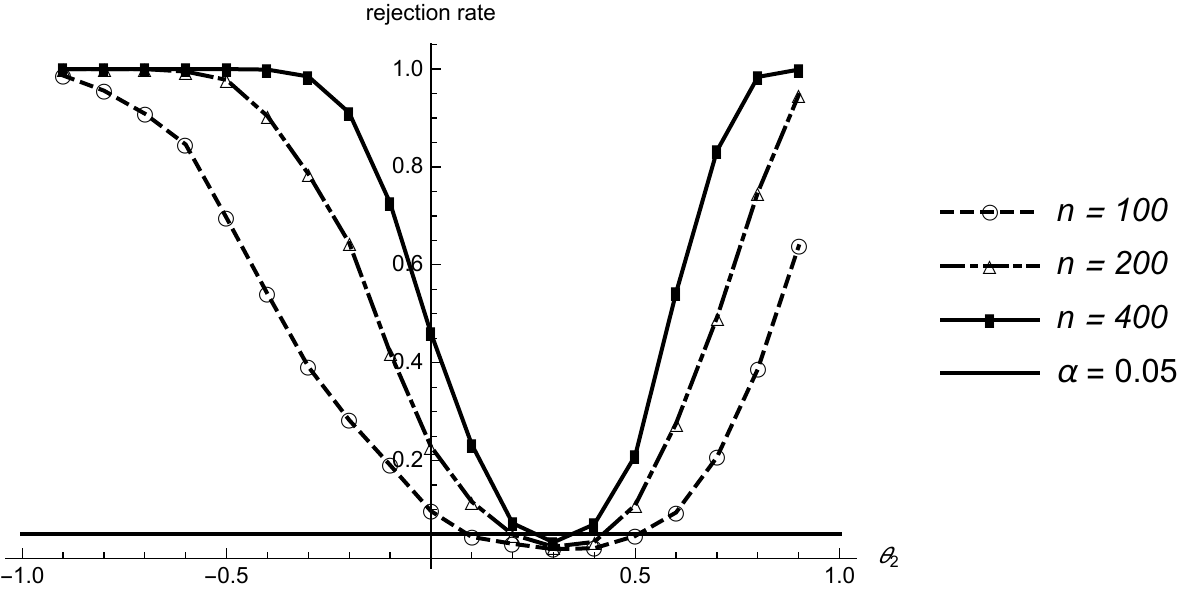}&
  \includegraphics[width=0.5\hsize]{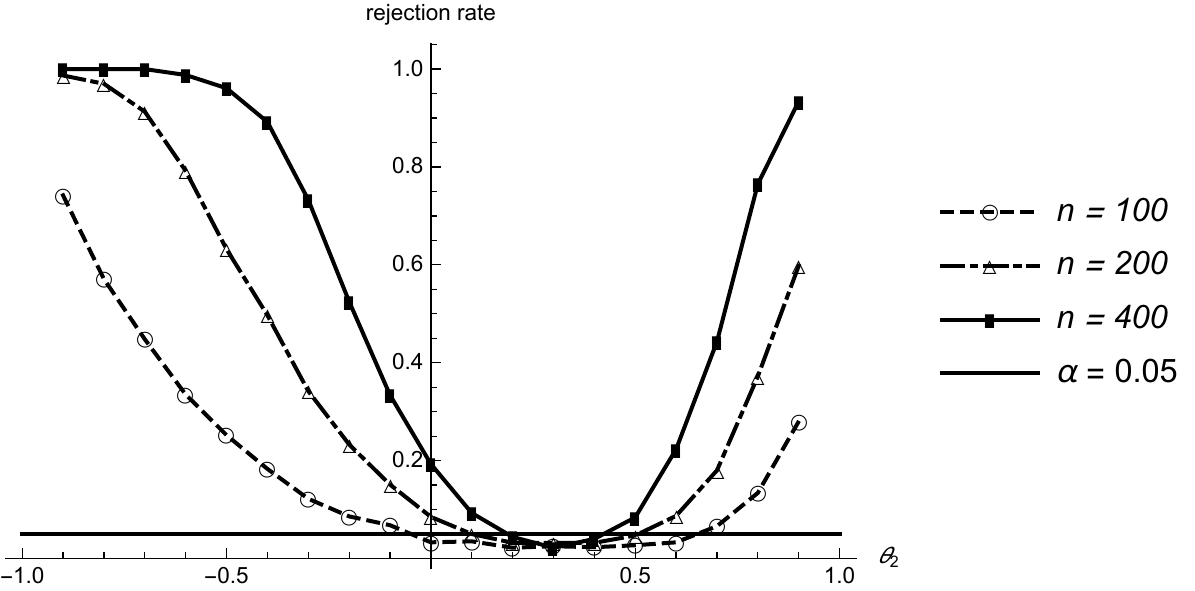}
\end{tabular}
\end{figure}

\begin{figure}[H]
\centering
\caption{\it  Simulated rejection probabilities of the ELR test \eqref{test} in the AR(1) model
with $t$-distributed innovations. } \label{fig:2}
\medskip
\medskip

\begin{tabular}{cc}
(a) $\theta_1 = 0.3$, $r=0.5$ & (b) $\theta_1 = 0.3$, $r=0.8$\\
  \includegraphics[width=0.5\hsize]{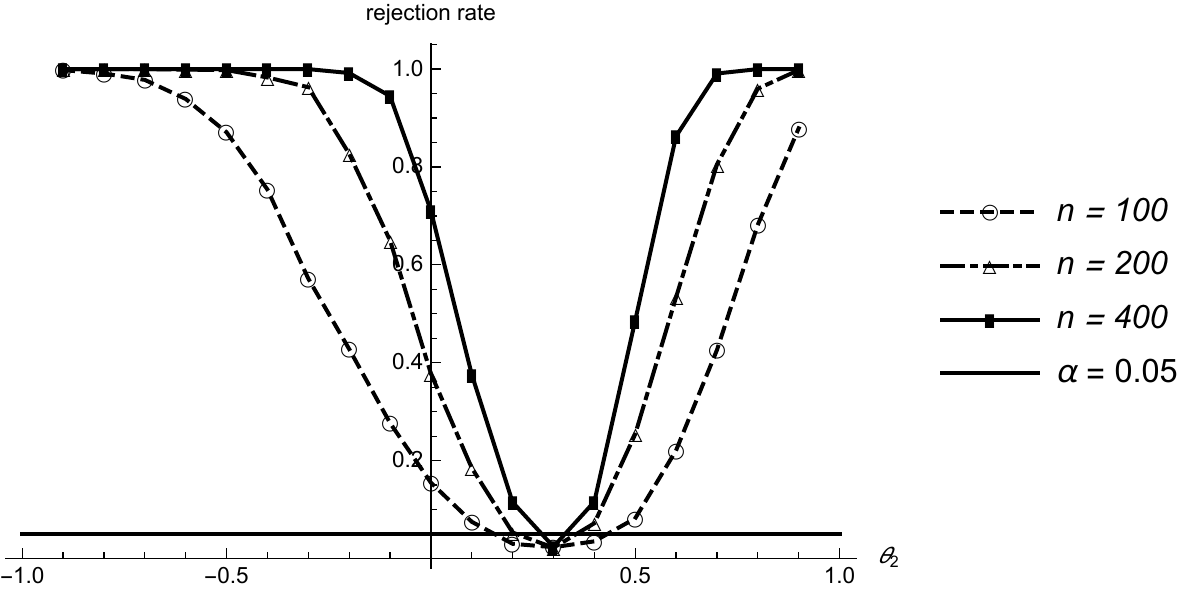}&
  \includegraphics[width=0.5\hsize]{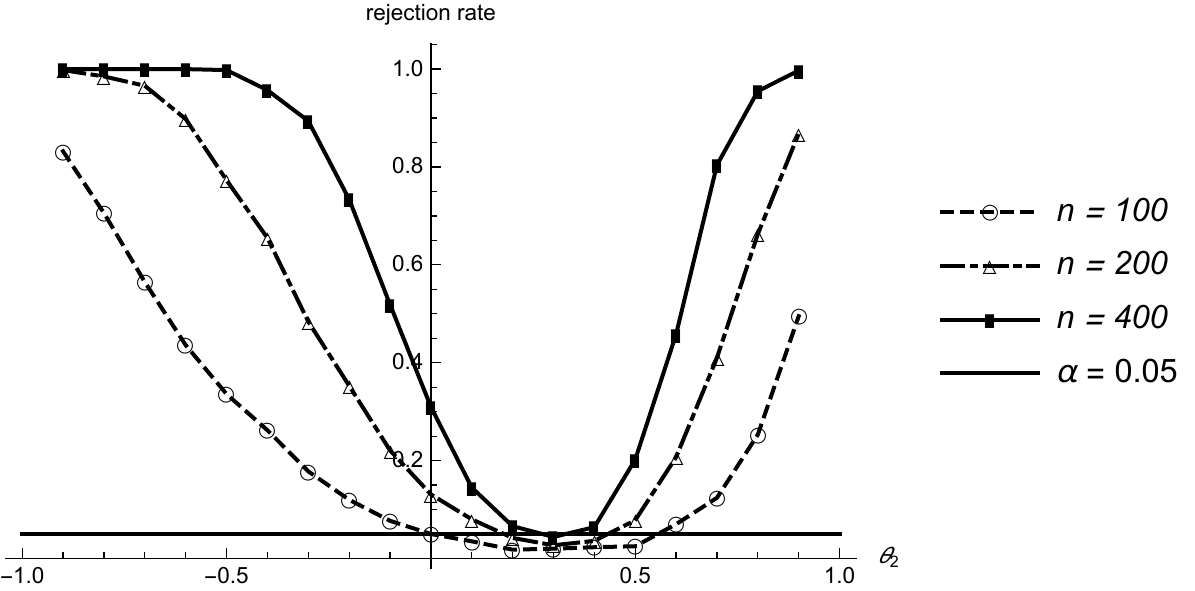}
\end{tabular}

\end{figure}

\begin{figure}[H]
\centering
\caption{\it  Simulated rejection probabilities of the ELR test \eqref{test} in the AR(1) model
with Cauchy distributed innovations. }
\label{fig:3}
\medskip
\medskip

\begin{tabular}{cc}
(a) $\theta_1 = 0.3$, $r=0.5$ & (b) $\theta_1 = 0.3$, $r=0.8$\\
  \includegraphics[width=0.5\hsize]{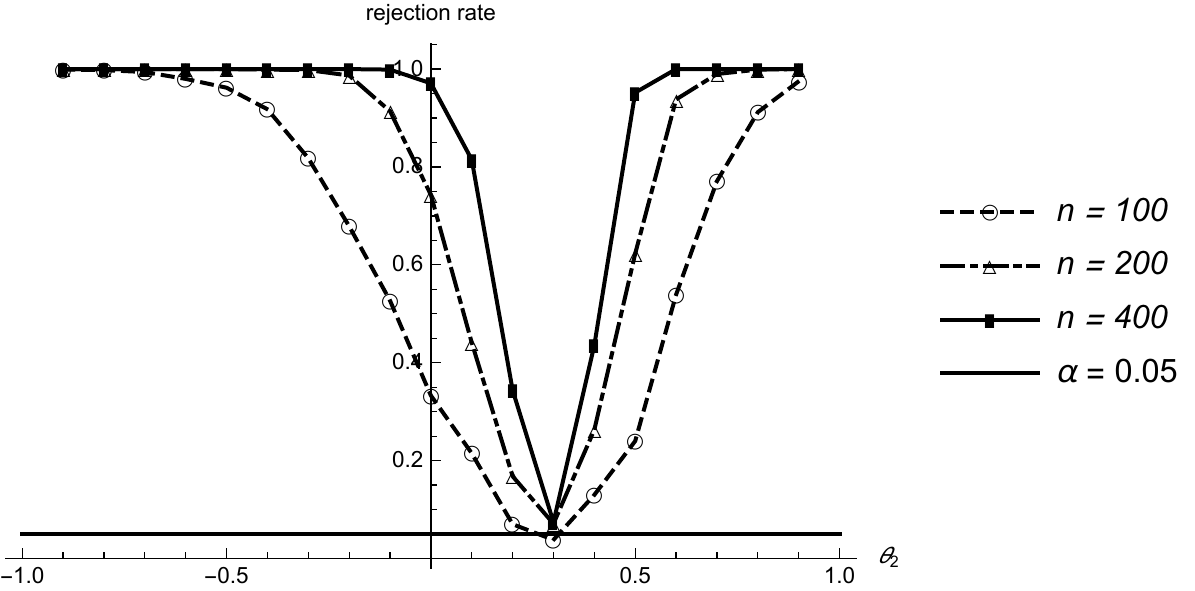}&
  \includegraphics[width=0.5\hsize]{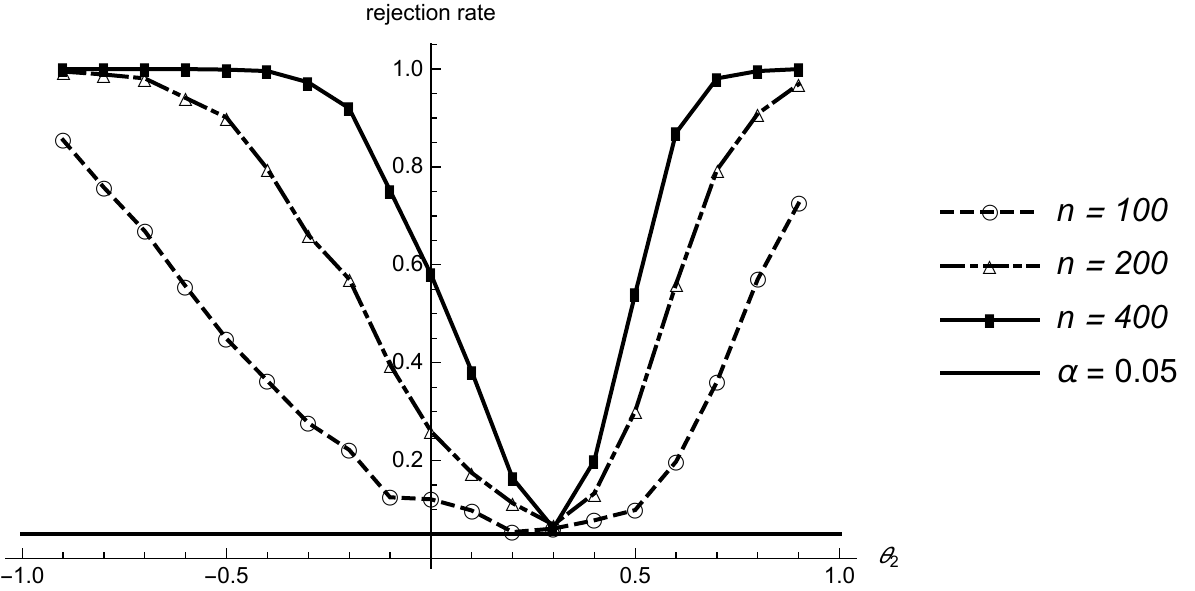}
\end{tabular}
\end{figure}

In the second part of this section
we compare   the new test defined by \eqref{test} with
the CUSUM test in \cite{qu2008} which uses quantile regression.
The test statistic for the median in \cite{qu2008} is defined by
\begin{equation} \label{qu}
{\rm SQ}_{0.5}
=
\sup_{\lambda  \in [0, 1]} \|H_{\lambda, n} (\hat{\beta}) - \lambda H_{1, n}(\hat{\beta}) \|,
\end{equation}
where $\|\cdot\|$ is the sup norm, $\hat{\beta}$ is the median regressor,
\[
H_{\lambda, n} = (\mathcal{X}' \mathcal{X})^{-1/2} \sum_{i = 1}^{[\lambda n]} |y_t - X_{t-1}' \hat{\beta}| X_{t - 1},
\]
and the matrix $\mathcal{X}$ is given by $\mathcal{X} = (X_1, \dots, X_n)'$.
In Figures \ref{fig:1c}-\ref{fig:3c}, we display the rejection probabilities of
the test based on the statistic $T_n$ in \eqref{tn}, $\tilde{T}_n$ in \eqref{stat} and ${\rm SQ}_{0.5}$ in \eqref{qu} for the hypothesis \eqref{eq:01},
where the nominal level is chosen as $\alpha = 0.05$.
The horizontal and vertical axes show, respectively, the values of $\theta_2$
and the rejection rate of the hypothesis $H:\theta_1 = \theta_2$
at this point ($\theta_1$ is fixed as 0.3).
The distribution of the innovation process is a standard normal distribution (Figure \ref{fig:1c}),
a $t$-distribution with 2 degree of freedom (Figure \ref{fig:2c}) and a Cauchy distribution (Figure \ref{fig:3c}) and
the sample sizes are given by $n = 100$, $200$ and $400$ in each case.
Again we consider two different locations for the change point  $k^*$
corresponding to the  values $r = 0. 5$ and $r = 0.8$.

We observe that  all tests derived from  the  three statistics $T_n$ in \eqref{tn} (corresponding to the weight function $h(r)  \equiv 1$), $\tilde{T}_n$ in \eqref{stat} (corresponding to the weight function $h(r) = r(1-r)$) and ${\rm SQ}_{0.5}$ in \eqref{qu}  are slightly conservative and that the approximation of the nominal level improves with increasing sample size
[see Figure \ref{fig:1c}-\ref{fig:3c} for the value $\theta_2=\theta_1=0.3$]. The approximation is usually more accurate
for  $r=0.5$.

Next we compare the power of the different tests (i.e. $\theta_2 \not =\theta_1=0.3$)  for different distributions of the innovations.
In the case of Gaussian innovations all tests shows a similar behavior (see Figure \ref{fig:1c}) and only if the case $n=200$ and $r=0.8$
the ELR test based on the (unweighted) statistic $T_n$ shows a better performance as the tests based
on   $\tilde{T}_n$ and ${\rm SQ}_{0.5}$. Moreover, for Gaussian innovations all three tests show a remarkable robustness
against non-stationarity, that is $|\theta_2| = 1$.

In Figure \ref{fig:2c} we
display corresponding results for $t_2$-distributed innovations. The differences in the approximation of the nominal level are negligible
($\theta_2=\theta_1=0.3$).
If $r=0.5$ we do not observe substantial differences in the power
between the three tests (independently of the sample size). On the other hand, if $r=0.8$
the  tests based on ELR statistics  $\tilde{T}_n$ and $T_n$ yield larger rejection probabilities   than the test ${\rm SQ}_{0.5}$
(see the right part of Figure Figure \ref{fig:2c}). Interestingly the unweighted test based on $T_n$ shows a better performance
than the test based on $\tilde{T}_n$ in these cases. Again, all tests are robust with respect to non-stationarity.

Finally,  in Figure \ref{fig:3c} we display the rejection probabilities of the three tests for Cauchy distributed innovations, where
we again do not observe differences in the approximation of the nominal   level ($\theta_2=\theta_1=0.3$). On the other hand
the differences in power  between the tests based on ELR and quantile regression are remarkable. In all cases  the ELR
tests based on $T_n$ and $\tilde{T}_n$ have substantially more power  than the test based on ${\rm SQ}_{0.5}$.
The  ELR test based on the unweighted statistic ${T}_n$ shows a better performance than the ELR test based on
$\tilde T_n$. This superiority is less pronounced in the  case  $r = 0.5$ but clearly visible for  $r = 0.8$.
Finally, in contrast to the test based on $SQ_{0.5}$ the ELR tests based on  $T_n$ and $\tilde{T}_n$ are  robust against non-stationarity
 (i.e. $|\theta_2| = 1$) for Cauchy distributed innovations and clearly detect a change  in the parameters in these cases.

 \begin{figure}[H]
 \centering
 \caption{\it  Simulated rejection probabilities of various change point tests based on the statistics $T_n$, $\tilde{T}_n$ and SQ$_{0.5}$
defined in \eqref{tn}, \eqref{stat} and \eqref{qu}, respectively. The model is given by an AR(1) model
 with  normal distributed innovations. } \label{fig:1c}
\medskip
\medskip
\begin{tabular}{cc}
\multicolumn{2}{c}{(i) $n=100$}\\ (a) $r=0.5$ & (b)  $r=0.8$\\
  \includegraphics[width=0.5\hsize]{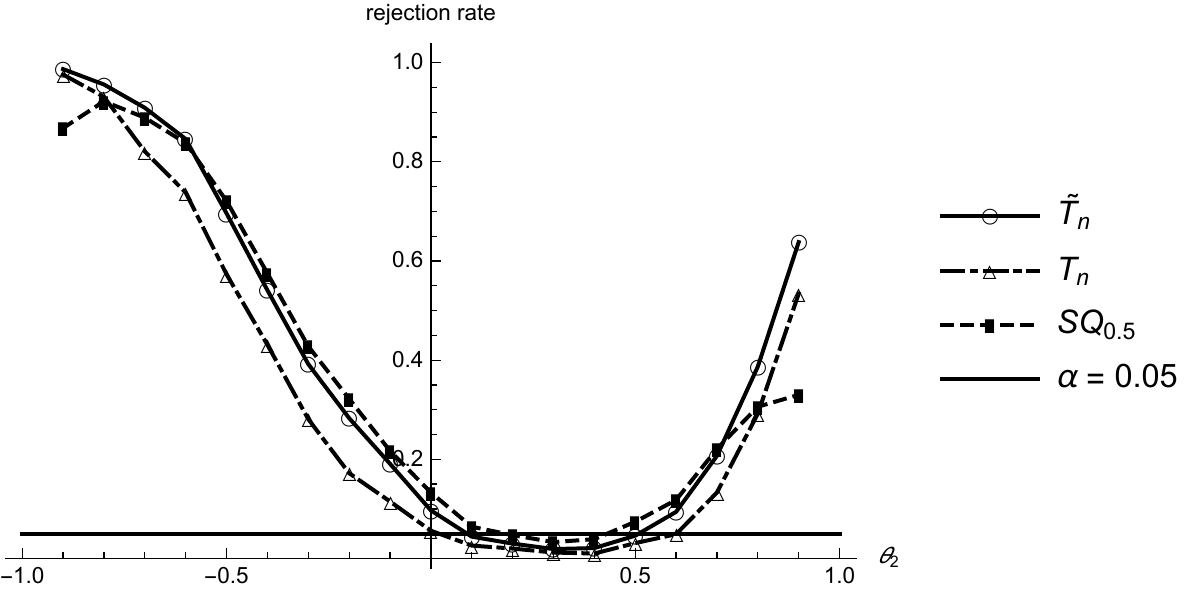}&
  \includegraphics[width=0.5\hsize]{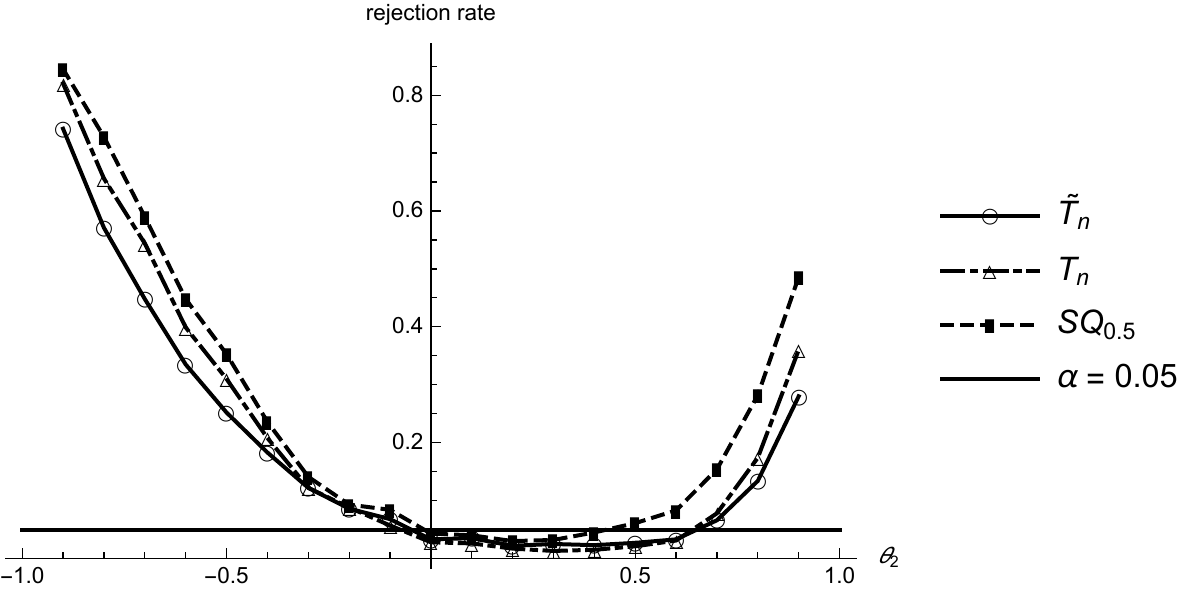}\\ \medskip\\
\multicolumn{2}{c}{(ii) $n=200$}\\ (a) $r=0.5$ & (b)  $r=0.8$\\
 \includegraphics[width=0.5\hsize]{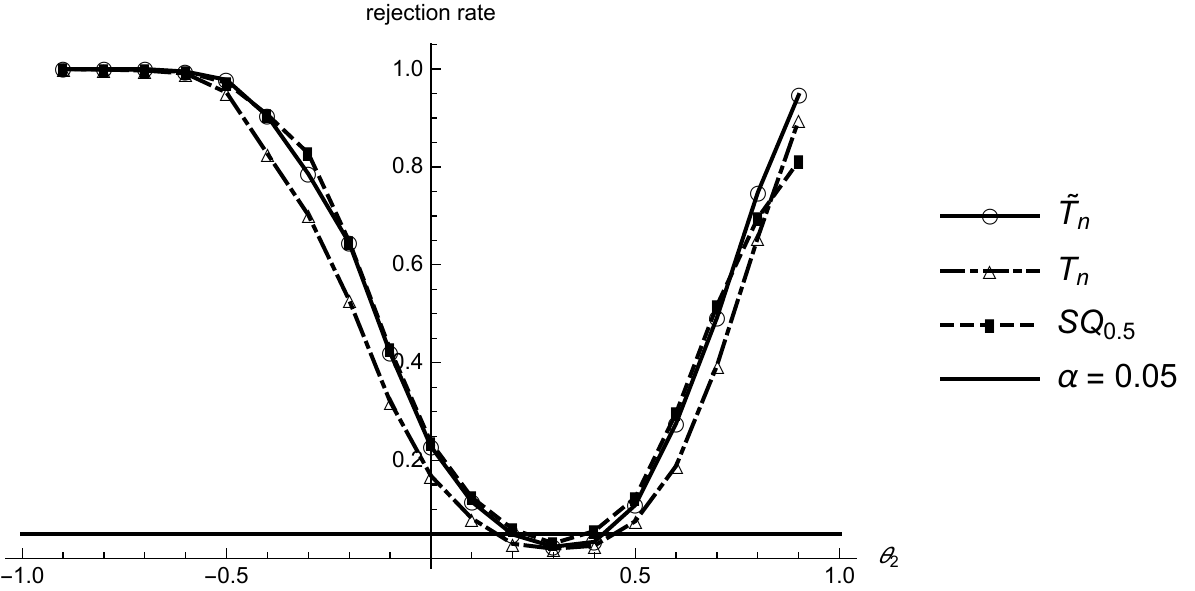}&
 \includegraphics[width=0.5\hsize]{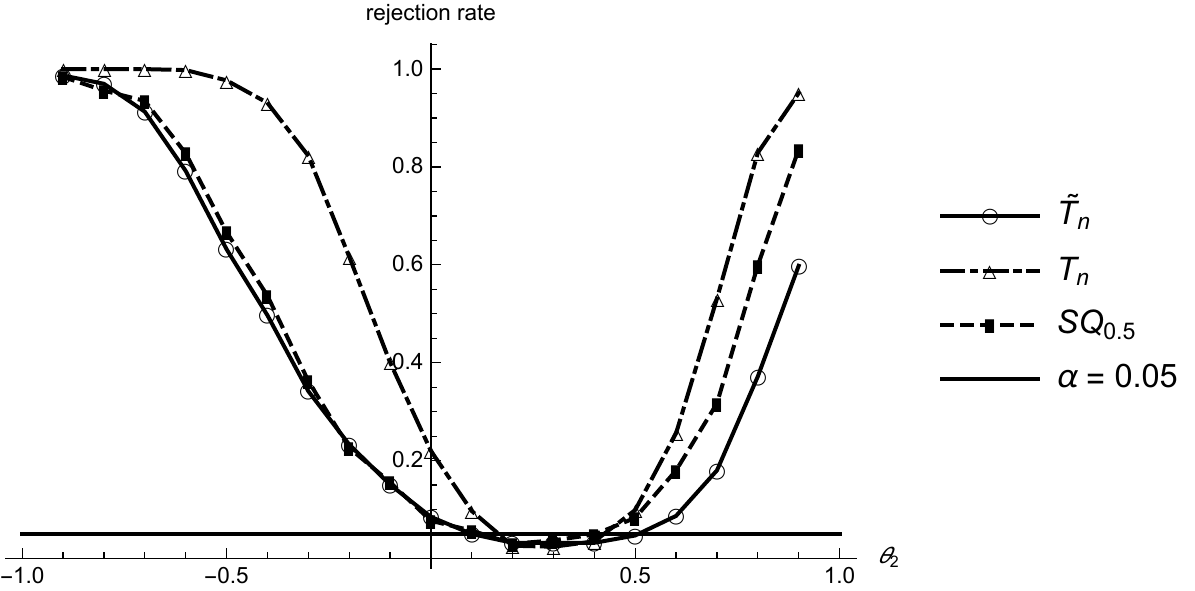}\\ \medskip\\
\multicolumn{2}{c}{(iii) $n=400$}\\ (a) $r=0.5$ & (b)  $r=0.8$\\
 \includegraphics[width=0.5\hsize]{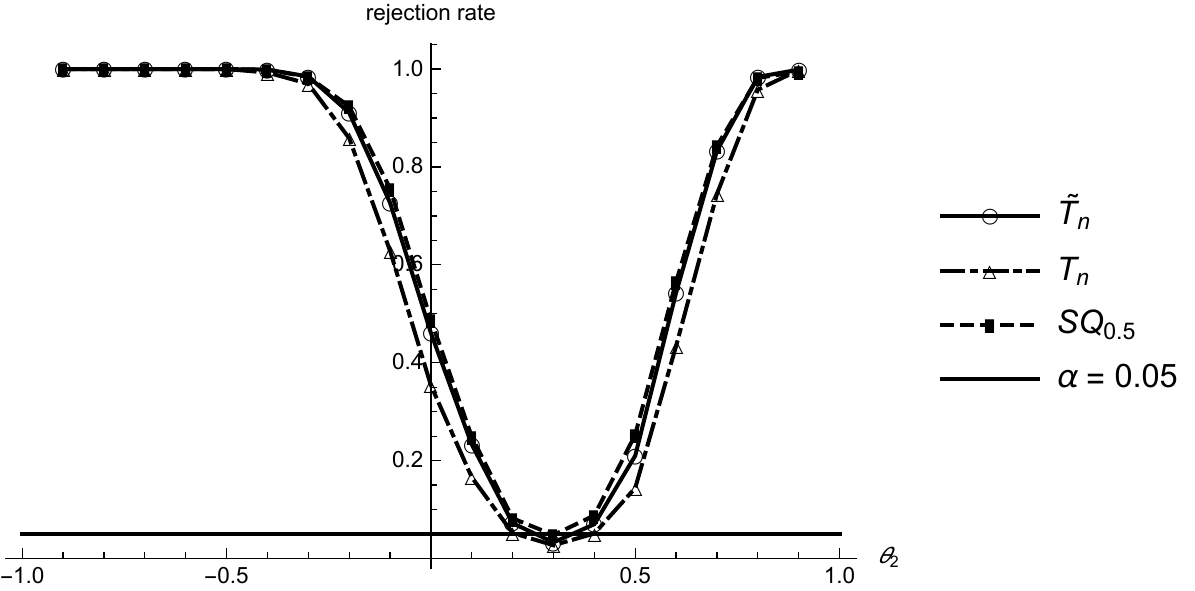}&
 \includegraphics[width=0.5\hsize]{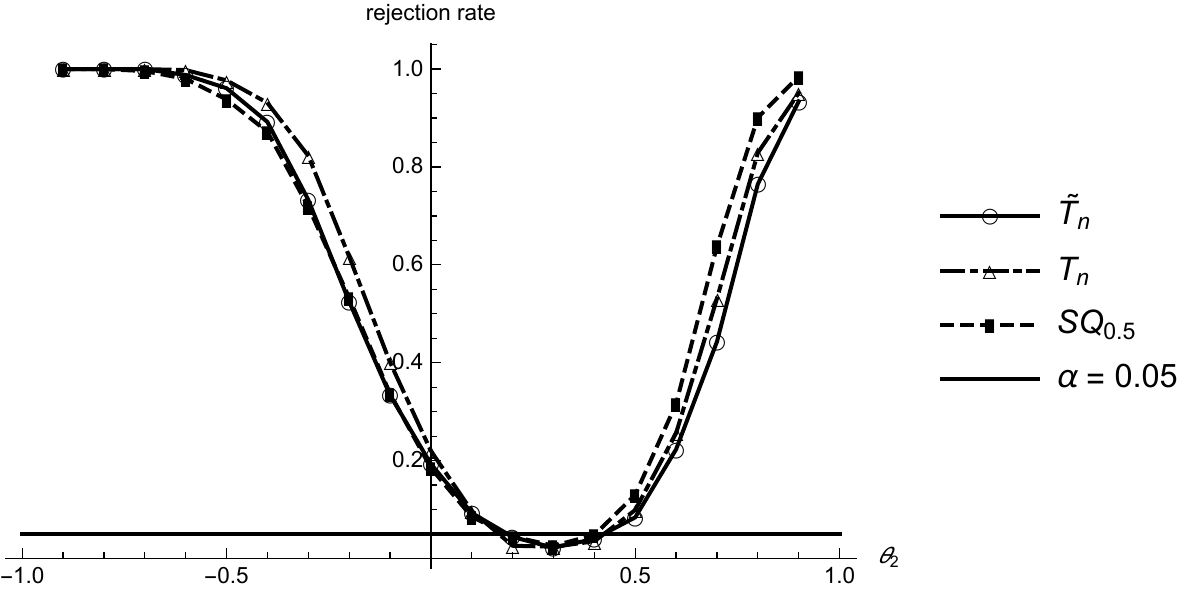}\\
\end{tabular}
\end{figure}

\begin{figure}[H]
\centering
 \caption{\it  Simulated rejection probabilities of various change point tests based on the statistics $T_n$, $\tilde{T}_n$ and SQ$_{0.5}$
defined in \eqref{tn}, \eqref{stat} and \eqref{qu}, respectively. The model is given by an AR(1) model
 with  $t_2$-distributed innovations. } \label{fig:2c}
\medskip
\medskip
\begin{tabular}{cc}
\multicolumn{2}{c}{(i) $n=100$}\\ (a) $r=0.5$ & (b)  $r=0.8$\\
 \includegraphics[width=0.5\hsize]{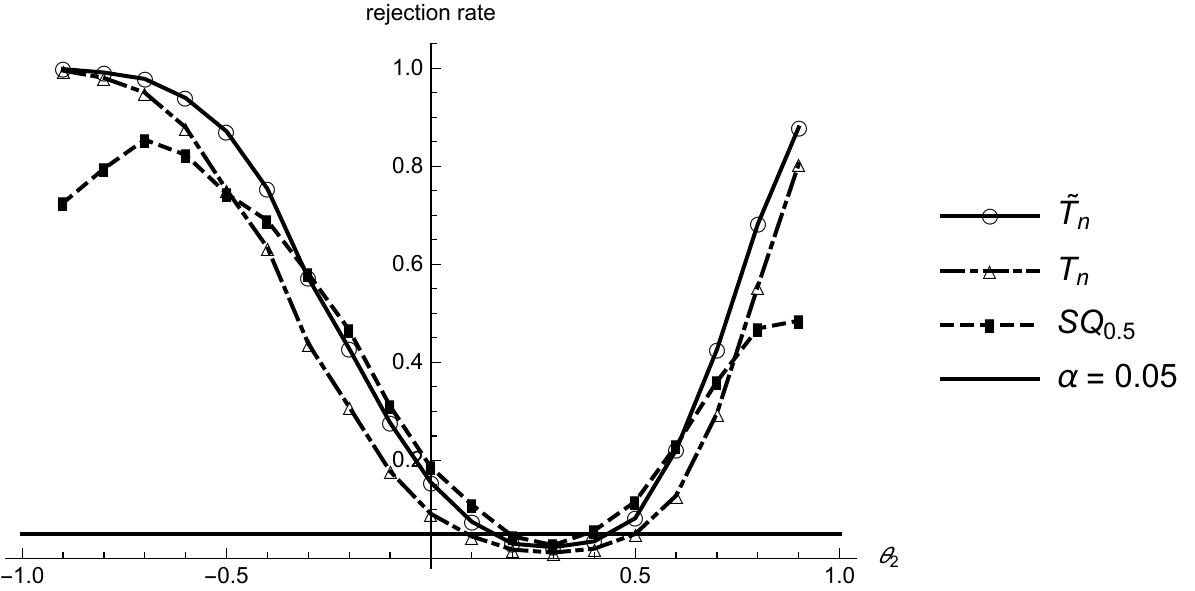}&
 \includegraphics[width=0.5\hsize]{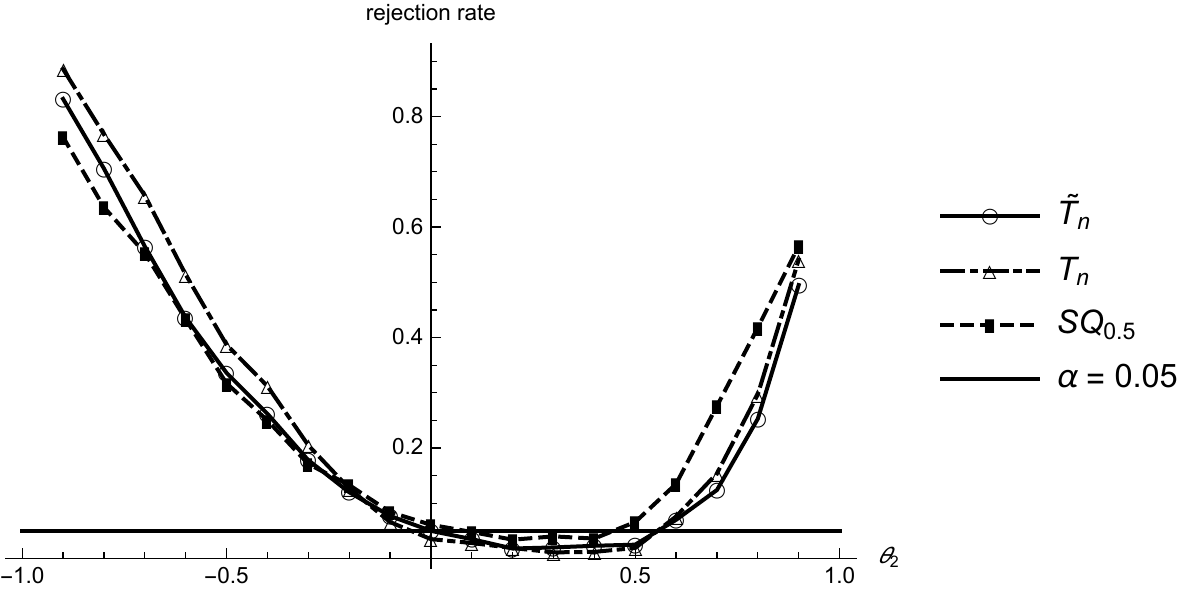}\\ \medskip\\
\multicolumn{2}{c}{(ii) $n=200$}\\ (a) $r=0.5$ & (b)  $r=0.8$\\
 \includegraphics[width=0.5\hsize]{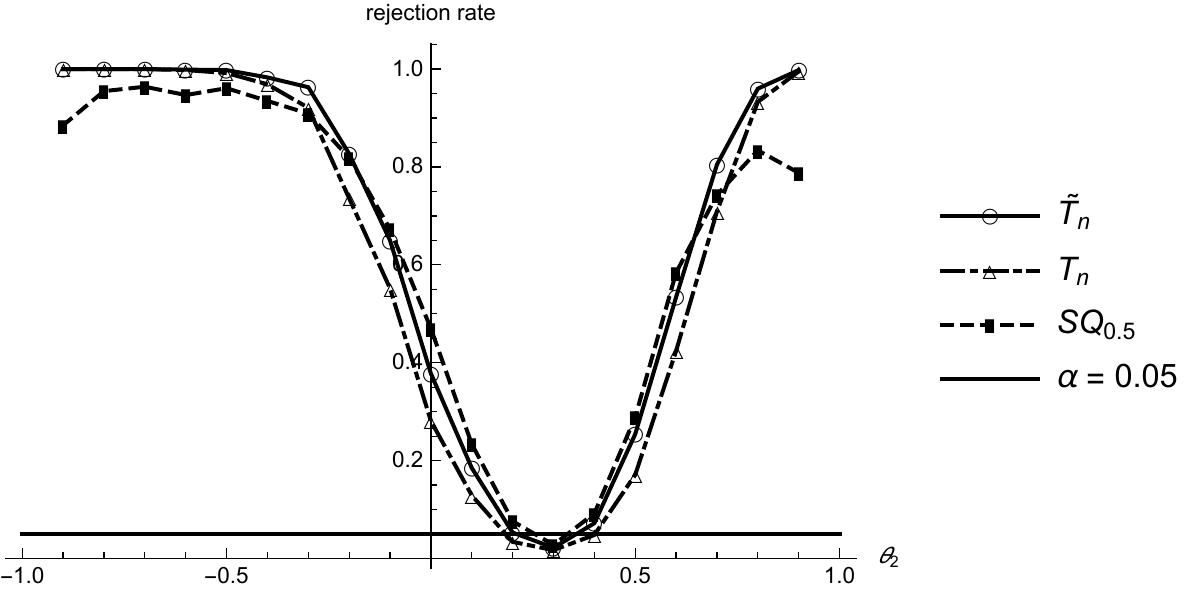}&
 \includegraphics[width=0.5\hsize]{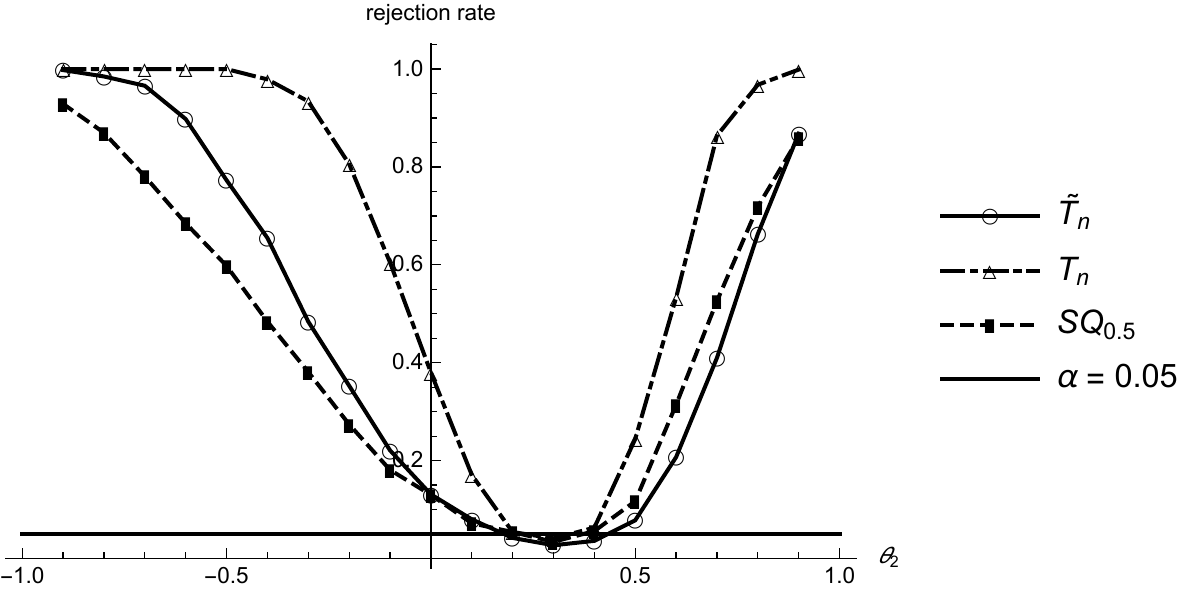}\\ \medskip\\
\multicolumn{2}{c}{(iii) $n=400$}\\ (a) $r=0.5$ & (b)  $r=0.8$\\
 \includegraphics[width=0.5\hsize]{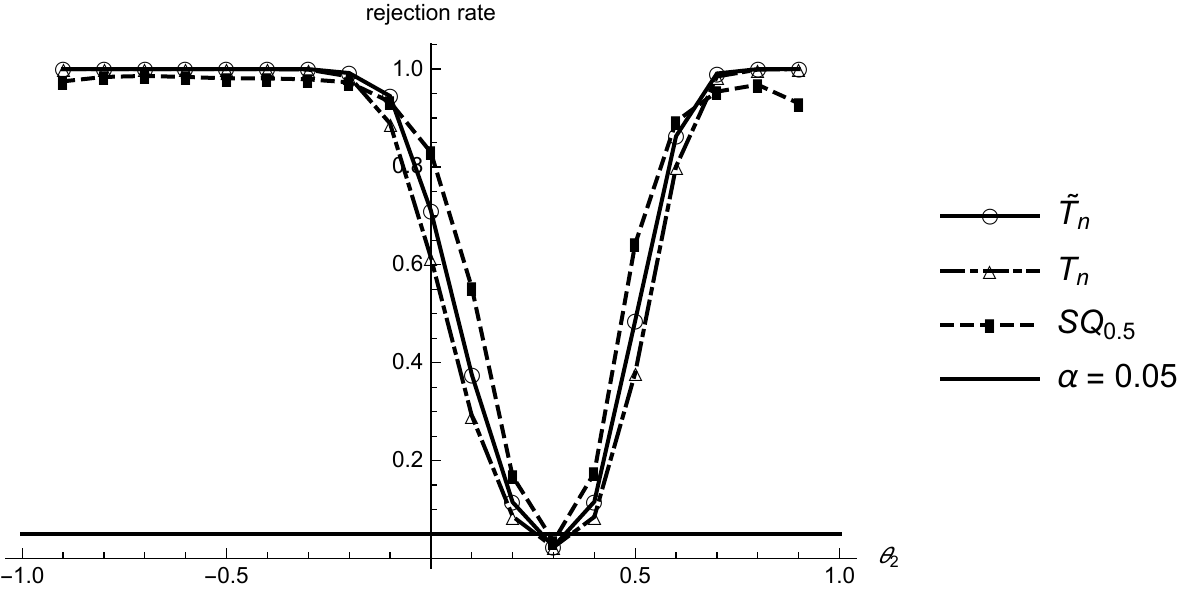}&
 \includegraphics[width=0.5\hsize]{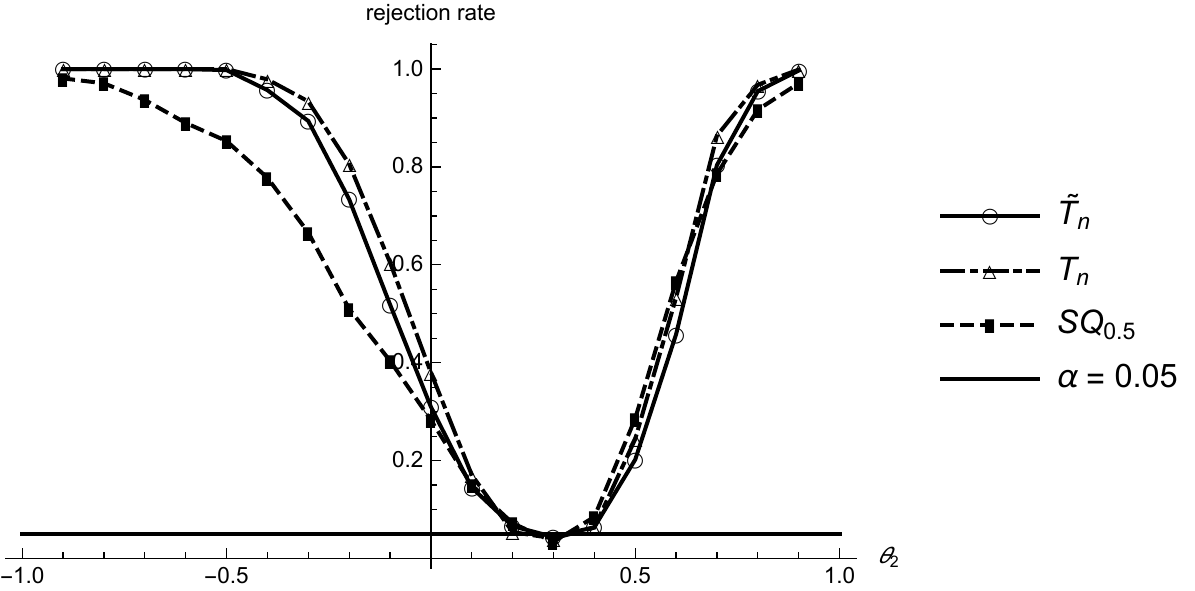}\\
\end{tabular}
\end{figure}

\begin{figure}[H]
\centering
\caption{\it  Simulated rejection probabilities of various change point tests based on the statistics $T_n$, $\tilde{T}_n$ and SQ$_{0.5}$
defined in \eqref{tn}, \eqref{stat} and \eqref{qu}, respectively. The model is given by an AR(1) model
 with  Cauchy distributed innovations. } \label{fig:3c}
\medskip
\medskip
\begin{tabular}{cc}
\multicolumn{2}{c}{(i) $n=100$}\\ (a) $r=0.5$ & (b)  $r=0.8$\\
 \includegraphics[width=0.5\hsize]{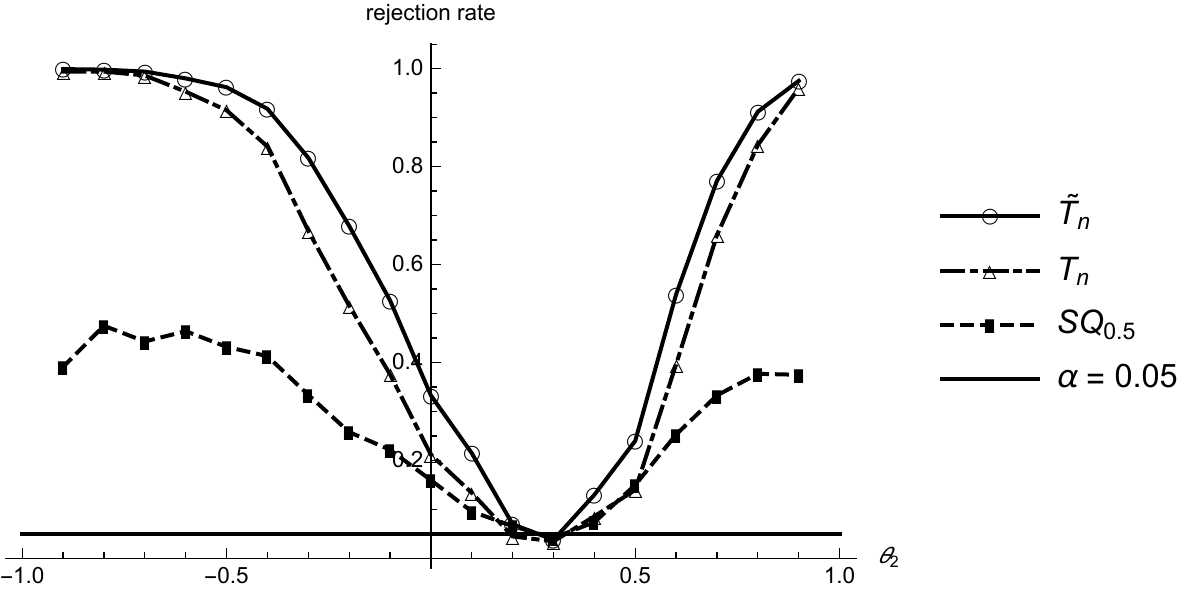}&
 \includegraphics[width=0.5\hsize]{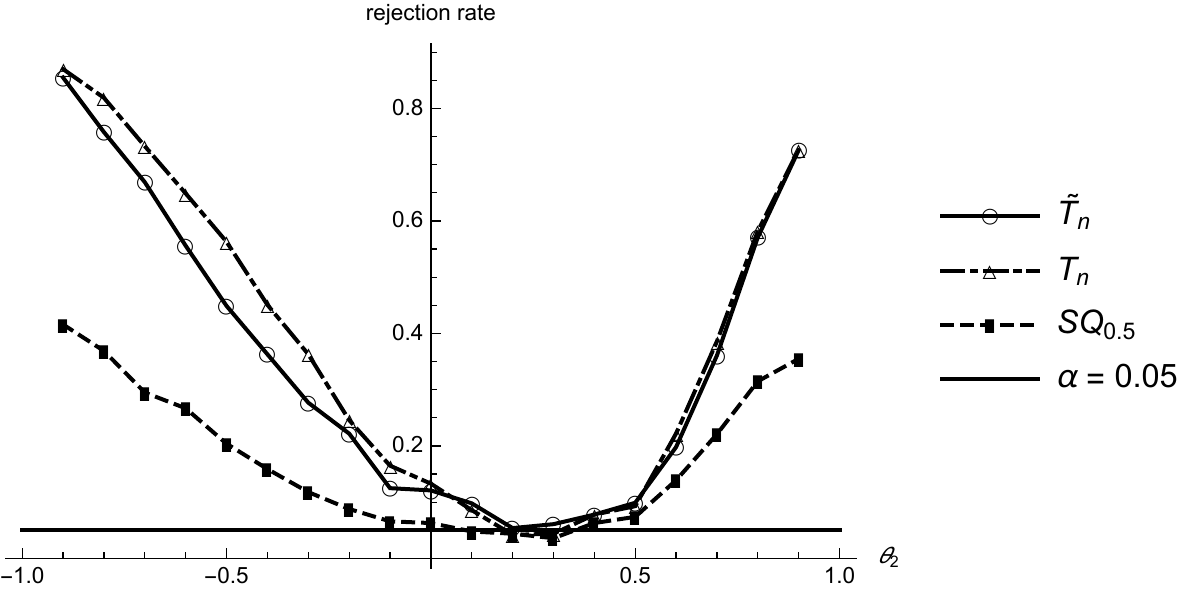}\\ \medskip\\
\multicolumn{2}{c}{(ii) $n=200$}\\ (a) $r=0.5$ & (b)  $r=0.8$\\
 \includegraphics[width=0.5\hsize]{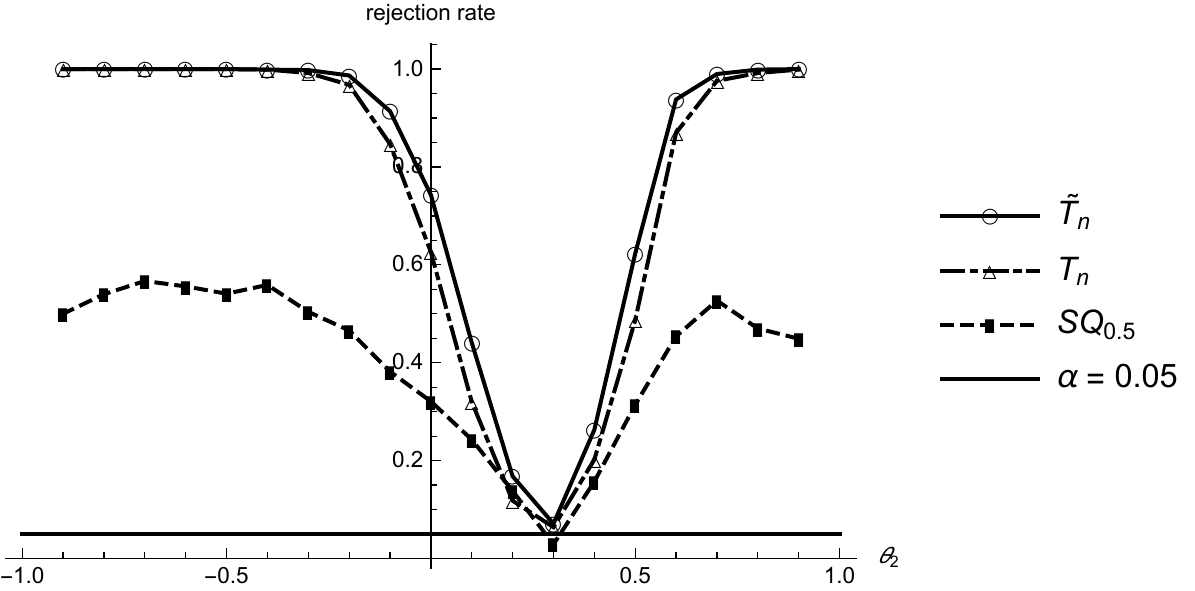}&
 \includegraphics[width=0.5\hsize]{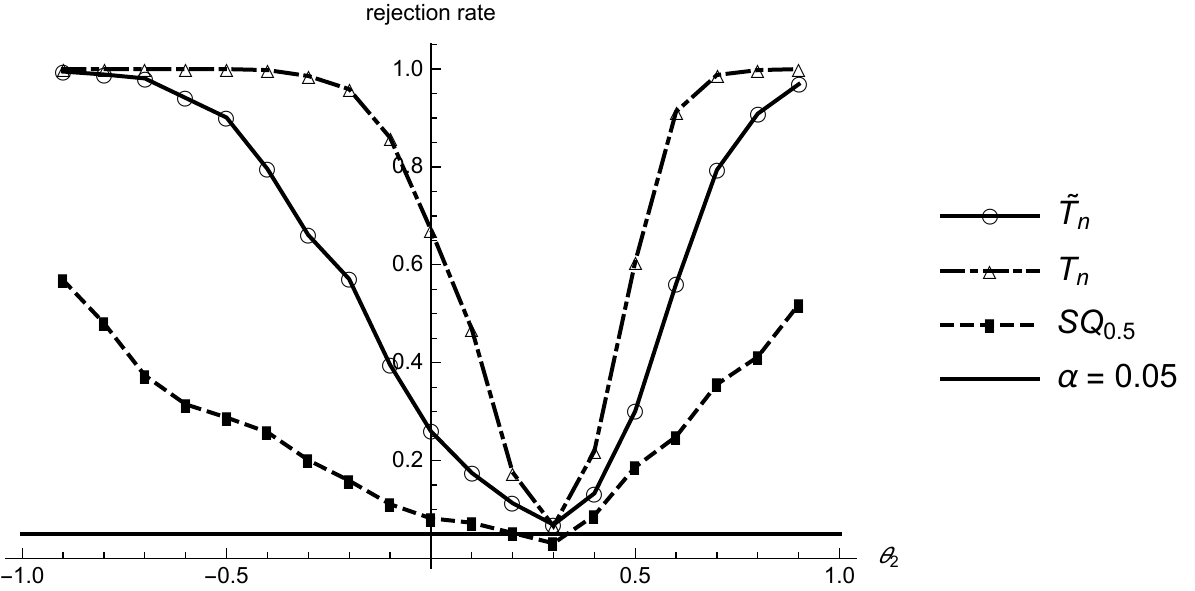}\\ \medskip\\
\multicolumn{2}{c}{(iii) $n=400$}\\ (a) $r=0.5$ & (b)  $r=0.8$\\
 \includegraphics[width=0.5\hsize]{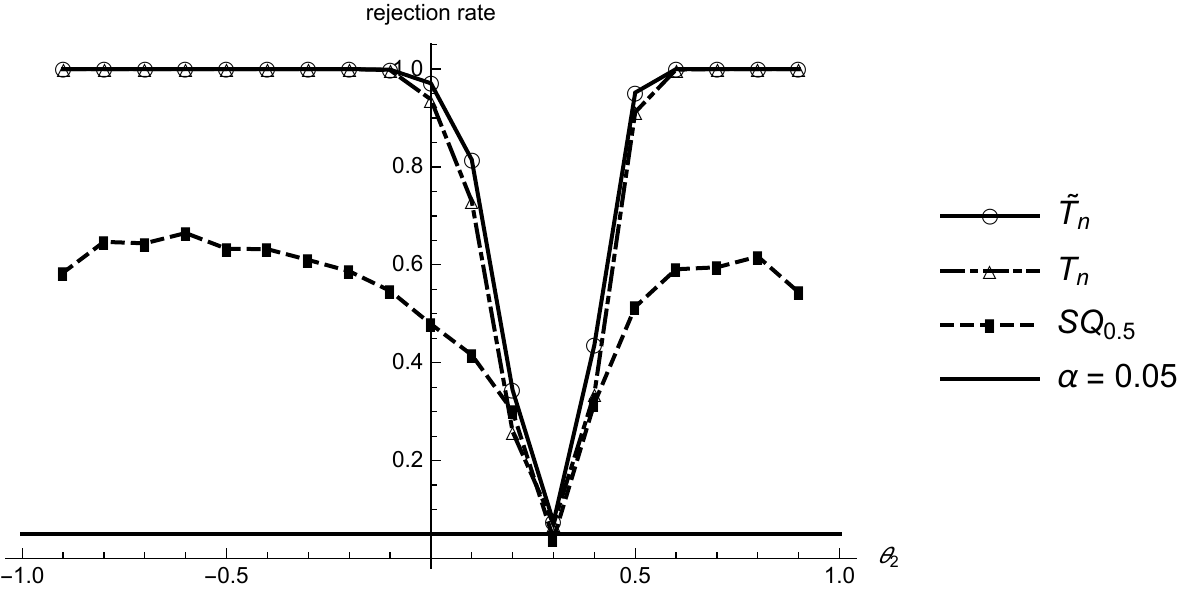}&
 \includegraphics[width=0.5\hsize]{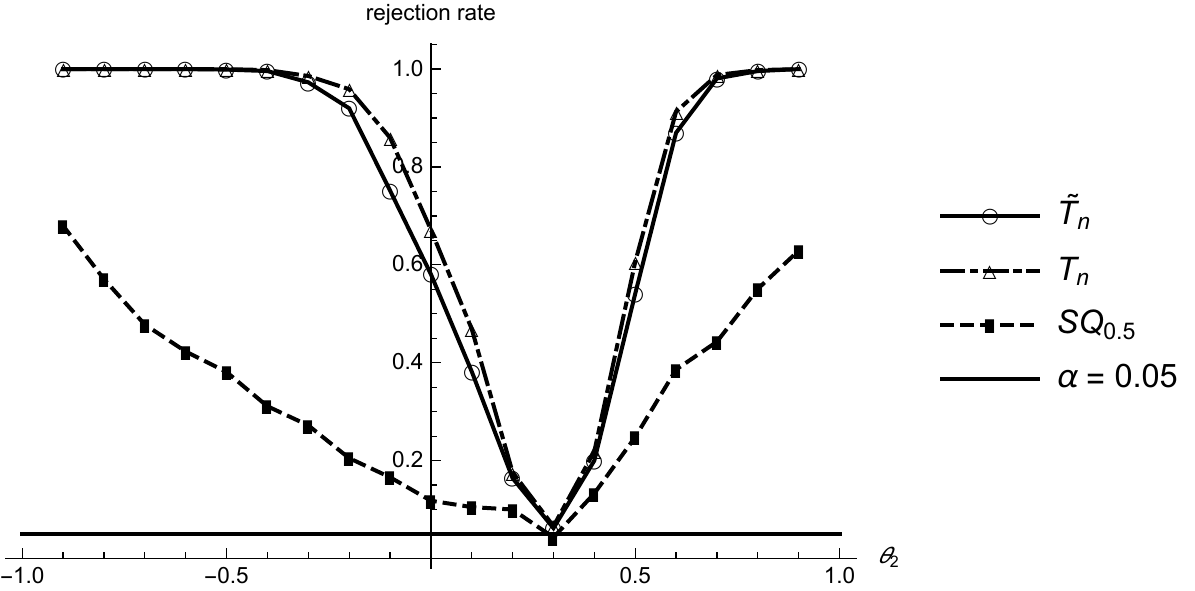}\\
\end{tabular}
\end{figure}

\section{Proofs} \label{sec5}
This section gives rigorous proofs of all  results in this paper.
In what follows, $C$ will denote a generic positive constant that varies in different places.
 ``with probability approaching one" will be abbreviated as w.p.a.1.
Moreover, we use the following notations throughout this section:
\begin{align}
&g_i(\beta) = g(\mathscr{Y}_{i}^{p} , \beta), \quad g(\beta) = \E[g(\mathscr{Y}_{i}^{p} , \beta)],\notag\\
&\hat{P}_k^1(\beta,\lambda) = \frac{1}{k}\sum_{i=1}^{k}\log\{1-\lambda'g_i(\beta)\},\notag\\
&\hat{P}^2_{n, k}(\beta,\eta) = \frac{1}{n-k}\sum_{j=k+1}^n\log\{1-\eta'g_j(\beta)\},\notag\\
&\hat\Lambda^1_{k}(\beta) = \left\{\lambda\in\R^m: |\lambda'g_i(\beta)| <1
\text{ for all $i=1,\ldots,k$}\right\},\notag\\
&\hat\Lambda^2_{n, k}(\beta) = \left\{\eta\in\R^m: |\eta'g_j(\beta)| <1
\text{ for all $j=k+1,\ldots,n$}\right\},\notag\\
&\hat g(\beta) = \frac{1}{n}\sum^{n}_{i=1}g(\mathscr{Y}^{p}_i, \beta)
= \frac{1}{n} \sum_{i = 1}^n g_i(\beta),\notag\\
&\hat{g}_k^1(\beta) = \frac{1}{k}\sum_{i=1}^k g_i(\beta)\quad
\text{ and }\quad
\hat{g}^2_{n,k}(\beta) = \frac{1}{n-k}\sum_{j=k+1}^n g_j(\beta).\notag
\end{align}

\subsection{Proof of Theorem \ref{thm:1}}
We start proving several auxiliary results which are required in the proof of  Theorem \ref{thm:1}.

\begin{lemma}\label{lem:A1}
Suppose that Assumption \ref{ass:4} (i) holds.
For $1/\gamma<\zeta<1/2$, let
\[
\Lambda_{n, k} =\{ (\lambda, \eta)\in\R^{2 m} : \|\lambda\|\leq C{k}^{-\zeta},
\quad \|\eta\| \leq C(n - k)^{-\zeta}  \}.
\]
Then, as $n \to \infty$, we have
\[
\sup_{\beta\in{\cal B},\lambda\in\Lambda_{n, k^*}}\max_{1\leq i\leq k^*}|\lambda' g_i(\beta)| \plim 0,\quad
\sup_{\beta\in{\cal B},\eta\in\Lambda_{n, k^*}}\max_{k^*+1\leq j\leq n}|\eta' g_j(\beta)| \plim 0.
\]
Also, $\Lambda_{n, k^*} \subset \hat\Lambda^1_{k^*}(\beta) \times \hat\Lambda^2_{n, k^*}(\beta)$ for all $\beta\in{\cal B}$ w.p.a.1.
\end{lemma}

\begin{proof}
Let $b_i = \sup_{\beta\in{\cal B}}\|g_i(\beta)\|$.
By Assumption \ref{ass:4} (i),
we can choose $\gamma>2$ such that
$K=\E[b_1^\gamma]^{1/\gamma}$ is finite.
Then, for any $\delta>0$, we can define $M(\delta)=K/\delta^{1/{\gamma}}$ and obtain
\begin{align}
P\Big( \max_{1\leq i\leq k^*}b_i \geq M(\delta){k^*}^{1/\gamma} \Big)
 &\leq \sum_{i=1}^{k^*} P\left(b_i \geq M(\delta){k^*}^{1/\gamma} \right)
 = \sum_{i=1}^{k^*} P\left(b_i^\gamma \geq M(\delta)^\gamma {k^*} \right)\notag\\
 &\leq \sum_{i=1}^{k^*} \frac{\E[b_i^\gamma]}{M(\delta)^\gamma {k^*}} = \delta.\notag
\end{align}
Consequently, $\max_{i}b_i = O_p({k^*}^{1/\gamma})$ and by the
  Cauchy-Schwartz inequality we have
\begin{equation*}
\sup_{\beta\in{\cal B},\lambda\in\Lambda_{n, k^*}}
\max_{1\leq i\leq k^*}|\lambda' g_i(\beta)|
 \leq \sup_{\lambda\in\Lambda_{n, k^*}}
 \|\lambda\|\max_{1\leq i\leq k^*}b_i\notag =O_p({k^*}^{-\zeta+1/\gamma}),
\end{equation*}
which implies
\[
\sup_{\beta\in{\cal B},\lambda\in\Lambda_{n, k^*}}\max_{1\leq i\leq k^*}|\lambda' g_i(\beta)| \plim 0.
\]
Similarly, it follows that
\[
\sup_{\beta\in{\cal B},\eta\in\Lambda_{n, k^*}}\max_{k^*+1\leq j\leq n}|\eta' g_j(\beta)|\plim 0.
\]
Therefore, $\Lambda_{n, k^*} \subset \hat\Lambda^1_{k^*}(\beta) \times \hat\Lambda^2_{n, k^*}(\beta)$ for all $\beta\in{\cal B}$ w.p.a.1, which completes the proof of Lemma \ref{lem:A1}.
\end{proof}

\medskip

\begin{lemma}\label{lem:A2}
Suppose that Assumptions \ref{ass:1} -- \ref{ass:4} hold,
and there exists a sequence
$\{\overline{\beta}_{n, k^*}\}\subset{\cal B}$ such that
$$\overline{\beta}_{n, k^*} \plim \beta_0, \
 \hat{g}_{k^*}^1(\overline{\beta}_{n, k^*}) = O_p({k^*}^{-1/2}) \mbox{ and }
 \hat g_{n, k^*}^2(\overline{\beta}_{n, k^*}) = O_p((n-k^*)^{-1/2})$$
  as $n\to\infty$.
Denote $\overline{\beta}_{n, k^*}$ by $\overline \beta$.
Then, under $H_0$,
\[
\overline\lambda := \arg\max_{\lambda\in\hat\Lambda^1_{k^*}(\overline{\beta})}
\hat{P}_{k^*}^1(\overline{\beta},\lambda)\quad\text{ and }\quad
\overline\eta := \arg\max_{\eta\in\hat\Lambda^2_{n, k^*}(\overline{\beta})}
\hat P_{n, k^*}^2(\overline{\beta}, \eta)\notag
\]
exist w.p.a.1. Moreover, as $n\to \infty $ we have
\begin{eqnarray*}
\overline\lambda&=&O_p({k^*}^{-1/2}),~~\overline\eta = O_p((n-k^*)^{-1/2}), \\
\hat{P}_{k^*}^1(\overline{\beta},\overline\lambda) &=& O_p({k^*}^{-1}),~~
\hat P_{n, k^*}^2(\overline{\beta},\overline\eta) = O_p((n-k^*)^{-1}).
\end{eqnarray*}
\end{lemma}

\begin{proof}
We only show the statement for $\overline \lambda$, the corresponding statement for $\overline \eta$ follows by similar arguments.
Since $\Lambda_{n, k^*}$ is a closed set, it follows that
\[
\check\lambda := \arg\max_{\lambda\in\Lambda_{n, k^*}}
\hat{P}_{k^*}^1(\overline\beta,\lambda)
\]
exists (note that $\hat P_{k^*}(\overline\beta, \lambda)$ is a concave function of $\lambda$). From Lemma \ref{lem:A1} it follows that $\hat{P}^1_{k^*}(\overline\beta,\lambda)$ is continuously
twice differentiable with respect to $\lambda$ w.p.a.1.
By a Taylor expansion at  $\lambda=0_m$,
there exists a point $\dot \lambda$ on the line joining $\check\lambda$ and $0_m$
such that
\begin{align}
0
 &= \hat{P}_{k^*}^1(\overline\beta,0_m)
 \leq \hat{P}_{k^*}^1(\overline \beta,\check\lambda)\notag\\
 &= -\check\lambda'\hat{g}_{k^*}^1(\overline\beta)
		+\frac{1}{2}\check\lambda'
		\Big[ \frac{1}{k^*}\sum^{k^*}_{i=1}\rho^1_{i}(\dot\lambda) g_i(\overline\beta) g_i(\overline\beta)' \Big]
			\check\lambda,\label{eq:A14}
\end{align}
where
$
\rho^1_{i}(\lambda) = -{1}/{(1-\lambda'g_i(\overline\beta))^2}
$.
Note that the  definition of $g_i(\beta)$ implies
\begin{align*}
g_i(\beta) g_i(\beta)'
 = \frac{1}{4}a^*(X_{i-1})a^*(X_{i-1})'
\end{align*}
for any $\beta \in \B$. By Lemma \ref{lem:A1}  we have $\rho^1_{i}(\dot\lambda)\geq -C$ uniformly with respect to
$i$ w.p.a.1.
Furthermore,  the ergodicity of $\{X_{t}:t\in\Z\}$ implies that the random variable
$$\hat\Omega^1_{k^*} := (4k^*)^{-1}\sum_{i=1}^{k^*} a^*(X_{i-1})a^*(X_{i-1})'$$
converges to $\Omega$ in probability.
Hence the minimum eigenvalue of $\hat\Omega^1_{k^*}$ is bounded away from 0 w.p.a.1.
and we obtain
 \begin{align} \label{eq:A13}
-\check\lambda'\hat{g}_{k^*}^1(\overline\beta)
		+\frac{1}{2}\check\lambda'
		\Big[ \frac{1}{k^*}\sum^{k^*}_{i=1} \rho_{i}^1(\dot{\lambda}) g_i(\overline\beta) g_i(\overline\beta)'\Big]
			\check\lambda
& \leq \|\check{\lambda}\|\, \|\hat{g}^1_{k^*}(\overline{\beta})\|
-
\frac{C}{2} \check{\lambda} \hat{\Omega}^1_{k^*} \check{\lambda} \nonumber\\
& \leq
\|\check \lambda\| \| \hat{g}^1_{k^*}(\overline{\beta})\| - C \|\check \lambda\|^2
\end{align}
w.p.a.1. Dividing both sides of  \eqref{eq:A13} by $\|\check\lambda\|$, we get
$$
 \|\check\lambda\| = O_p({k^*}^{-1/2})=o_p({k^*}^{-\zeta}),$$ and hence
$\check\lambda\in\Int(\hat\Lambda^1_{k^*})$ w.p.a.1.
Again by Lemma \ref{lem:A1}, the
 concavity of $\hat{P}_{k^*}^1(\overline\beta,\lambda)$
and the convexity of $\hat\Lambda^1_{k^*}(\overline\beta)$, it follows that $\overline\lambda
=\check\lambda$ exists w.p.a.1 and $\overline\lambda=O_p({k^*}^{-1/2})$.
These results also imply that
$\hat{P}_{k^*}^1(\overline\beta,\overline\lambda) = O_p({k^*}^{-1})$.
By   similar arguments,
we can show the corresponding results for  $\overline\eta$ and $\hat{P}_{n, k^*}^2(\overline\beta,\overline\eta)$.
\end{proof}

\medskip

Next, let us consider the estimator $\hat{\beta}_{n,k}$
of Theorem \ref{thm:1}. Recall that $\hat{\beta}_{n,k}$ is the minimizer of
\[
l_{n, k}(\beta, \beta)
=
k  \sup_{\lambda\in\hat\Lambda^1_{k }(\beta)}\hat P_k^1(\beta,\lambda)
+ (n - k ) \sup_{\eta\in\hat\Lambda^2_{n, k }(\beta)}\hat P^2_{n,k}(\beta,\eta).
\]
Let us define
\begin{equation} \label{eq:l5.2}
\hat P_{n, {k}}(\beta,\lambda,\eta):=
k \hat P_k^1(\beta,\lambda) + (n - {k}) \hat P^2_{n, k}(\beta,\eta)
\end{equation}
and
\begin{align} \label{eq:l5.3}
\hat\lambda_{n, k} := \arg\max_{\lambda\in\hat\Lambda^1_{k}(\hat{\beta}_{n,k})}
\hat{P}_k^1(\hat{\beta}_{n,k},\lambda),\quad
\hat\eta_{n, k} := \arg\max_{\eta\in\hat\Lambda^2_{n, k}(\hat{\beta}_{n,k})}
\hat P^2_{n, k}(\hat{\beta}_{n,k},\eta).
\end{align}

\begin{lemma}\label{lem:A3}
Suppose that Assumptions \ref{ass:1} -- \ref{ass:4} hold.
Then, under the null hypothesis $H_0$ of no change point we have
$$\hat{g}_{k^*}^1(\hat\beta_{n,k^*}) = O_p({k^*}^{-1/2}),~~
\hat g_{n, k^*}^2(\hat\beta_{n,k^*}) = O_p((n-{k^*})^{-1/2})
$$
as $n\to\infty$.
\end{lemma}

\begin{proof}
Define $\hat{\hat g}_{n,k}^l := \hat g^l(\hat{\beta}_{n,k})$ for $l = 1, 2$,
\begin{align}
\tilde\lambda_{n, k} := -{k}^{-1/2}{\hat{\hat g}^1_{n, k}}/{\|\hat{\hat g}^1_{n, k}\|},\quad
\tilde\eta_{n, k} := -(n - {k})^{-1/2}{\hat{\hat g}^2_{n, k}}/{\|\hat{\hat g}^2_{n, k}\|}\label{eq:A2},
\end{align} then it follows  from \eqref{eq:l5.3} that
\[
\hat P_k^1(\hat\beta_{n,k},\tilde\lambda_{n,k})
\leq \hat P_k^1(\hat\beta_{n,k},\hat\lambda_{n,k})
\quad\text{and}\quad
\hat P_{n,k}^2(\hat\beta_{n,k},\tilde\eta_{n,k})\leq
\hat P_{n,k}^2(\hat\beta_{n,k},\hat\eta_{n,k}),
\]
which implies the  inequality
\begin{align}
\hat P_{n,k}(\hat\beta_{n,k},\tilde\lambda_{n,k},\tilde\eta_{n,k})
\leq
\hat P_{n,k}(\hat\beta_{n,k},\hat\lambda_{n, k},\hat\eta_{n, k}).\label{eq:A4}
\end{align}
By  similar arguments as used in \eqref{eq:A14} and \eqref{eq:A13}  we have
\begin{align} \label{eq:A3}
\hat P_{n, {k^*}}(\hat{\beta}_{n, {k^*}},\tilde\lambda_{n, {k^*}},\tilde\eta_{n, {k^*}})
 \geq {k^* }^{1/2}\|\hat{\hat g}_{n, {k^*}}^1\| +(n-k^* )^{1/2}\|\hat{\hat g}_{n, {k^*}}^2\|-c_0
\end{align}
 w.p.a.1, where $c_0$ is the same constant as in the proof of
Lemma \ref{lem:A2}.
On the other hand, we have the following inequality:
\begin{align}
\hat P_{n,k}(\hat\beta_{n,k},\hat\lambda_{n,k},\hat\eta_{n,k})
 &= \inf_{\beta\in {\cal B}}
\sup_{\lambda\in\hat\Lambda_{k}^1(\hat\beta_{n,k}),\
\eta\in\hat\Lambda^2_{n,k}(\hat\beta_{n,k})}
\hat P_{n,k}(\beta,\lambda,\eta)\notag\\
 &\leq  \sup_{\lambda\in\hat\Lambda_k^1(\beta_0),\
\eta\in\hat\Lambda^2_{n,k}(\beta_0)}
\hat P_{n,k}(\beta_0,\lambda,\eta)\notag\\
 &\leq  k \sup_{\lambda\in\hat\Lambda_k^1(\beta_0)}
\hat P^1_k(\beta_0,\lambda)
+
(n-k )\sup_{\eta\in\hat\Lambda^2_{n,k}(\beta_0)}
\hat P_{n,k}^2(\beta_0,\eta).\label{eq:AA1}
\end{align}
Applying Lemma \ref{lem:A2} with $\overline\beta_{n, {k^*}} = \beta_0$ yields
\begin{align}
\sup_{\lambda\in\hat\Lambda_{k^*}^1(\beta_0)}
\hat P^1_{k^*}(\beta_0,\lambda) = O_p({k^*}^{-1}), \quad
\sup_{\eta\in\hat\Lambda^2_{n, {k^*}}(\beta_0)}
\hat P_{n, {k^*}}^2(\beta_0,\eta) = O_p((n- {k^*})^{-1}), \label{eq:AA2}
\end{align}
and from \eqref{eq:AA1} and \eqref{eq:AA2}, we get
\begin{align}
\hat P_{n, {k^*}}(\hat\beta_{n, {k^*}},\hat\lambda_{n, {k^*}},\hat\eta_{n, {k^*}})
&= O_p(1).\label{eq:A5}
\end{align}
Finally, from  \eqref{eq:A4}, \eqref{eq:A3} and \eqref{eq:A5},
we have
\begin{align*}
-c_0
\leq
-c_0 + {k^*}^{1/2}\|\hat{\hat g}_{n, {k^*}}^1\| +(n-{k^*} )^{1/2}\|\hat{\hat g}_{n, {k^*}}^2\| \leq \hat P_{n, {k^*}}(\hat\beta_{n, {k^*}},
\hat\lambda_{n, {k^*}},\hat\eta_{n, {k^*}}) = O_p(1),
\end{align*}
which implies
\[
\|\hat{\hat g}_{n, {k^*}}^1\| = \|\hat g_{k^*}^1(\hat{\beta}_{n, {k^*}})\| =
O_p({k^*}^{-1/2})\quad \text{and}\quad
\|\hat{\hat g}_{n, {k^*}}^2\| = \|\hat g_{n, {k^*}}^2(\hat{\beta}_{n, {k^*}})\| =
O_p((n-{k^*})^{-1/2}),
\]
establishing the assertion of Lemma \ref{lem:A3}.
\end{proof}

\medskip

\begin{proof}\renewcommand{\qedsymbol}{}[Proof of Theorem \ref{thm:1}] By Lemma \ref{lem:A3}  we have $\hat{g}(\hat{\beta}_{n, k^*}) = o_p(1)$.
Then, it follows from the triangular inequality and  uniform law of large numbers
that
\begin{align*}
\|g(\hat\beta_{n,{k^*}})\|
 &\leq \|g(\hat\beta_{n,{k^*}}) - \hat g(\hat\beta_{n,{k^*}})\| + \|\hat g(\hat\beta_{n,{k^*}})\|\notag\\
 &\leq \sup_{\beta\in{\cal B}}\|g(\beta) - \hat g(\beta)\| + \|\hat g(\hat\beta_{n,{k^*}})\|
 = o_p(1).
\end{align*}
Since $g(\beta)$ has a unique zero at $\beta_0$, the function $\|g(\beta)\|$ must be bounded
away from zero outside any neighborhood of $\beta_0$.
Therefore, $\hat\beta_{n, k^*}$ must be inside any neighborhood of $\beta_0$ w.p.a.1.\
and therefore, $\hat\beta_{n, k^*} \plim \beta_0$.

Next, we show that
$\hat\beta_{n,{k^*}}-\beta_0 = O_p(n^{-1/2})$.
As $k^* = rn$, by Lemma \ref{lem:A3}, we have
\begin{align*}
\hat g(\hat\beta_{n,{k^*}}) &= n^{-1}\left\{ {k^*} \hat g_{k^*}^1(\hat\beta_{n,{k^*}})
+(n-{k^*})\hat g_{n,{k^*}}^2(\hat\beta_{n,{k^*}})\right\}
 =O_p(n^{-1/2})
\end{align*}
and the   central limit theorem  implies
\[
\hat g(\beta_0) = O_p\left[n^{-1}\left\{{k^*}^{1/2} + (n-{k^*})^{1/2}\right\}\right]
=
O_p(n^{1/2}).
\]
Further,
\begin{align}
\|\hat g(\hat\beta_{n, k^*}) - \hat g(\beta_0) - g(\hat\beta_{n, k^*})\|\leq
(1+\sqrt{n}\|\hat\beta_{n, k^*}-\beta_0\|)o_p(n^{-1/2}),\label{eq:A7}
\end{align}
which yields
\begin{align*}
\|g(\hat\beta_{n, k^*})\|
 &\leq \| \hat g(\hat\beta_{n, k^*})-\hat g(\beta_0)-g(\hat\beta_{n, k^*}) \|
 + \|\hat g(\hat\beta_{n, k^*})\| + \|\hat g(\beta_0)\|\notag\\
 &= (1+\sqrt{n}\|\hat\beta_{n, k^*}-\beta_0\|)o_p(n^{-1/2}) +
 O_p\left[n^{-1}\left\{{k^*}^{1/2} + (n-{k^*})^{1/2}\right\}\right].
\end{align*}
Moreover,  similar arguments as given  in \cite{newey1994large} on page 2191,
 the differentiability of $\|g(\beta)\|$ and the estimate
$\|g(\hat\beta_n)\|\geq C\|\hat\beta_n-\beta_0\|$ w.p.a.1. show that
\[
\|\hat\beta_{n, k^*} - \beta_0\|
 = (1+\sqrt{n}\|\hat\beta_{n, k^*}-\beta_0\|)o_p(n^{-1/2}) +
O_p\left[n^{-1}\left\{{k^*}^{1/2} + (n-{k^*})^{1/2}\right\}\right],
\]
and hence
\begin{align}
\{1+o_p(1)\}\|\hat\beta_{n, k^*}-\beta_0\| = o_p(n^{-1/2})
+O_p\left[n^{-1}\left\{{k^*}^{1/2} + (n-{k^*})^{1/2}\right\}\right].
\label{eq:AA4}
\end{align}
If $k^*=r n$ the right-hand side of
\eqref{eq:AA4} is of order $O_p(n^{-1/2})$, which
  completes the proof of Theorem \ref{thm:1}.
\end{proof}

\subsection{Proof of Theorem \ref{thm:2}}
We first show that $\hat P_{n, k^*}(\beta,\lambda,\eta)$ in \eqref{eq:l5.2} is well approximated by some function
near its optima using a similar reasoning as in \cite{parente2011gel}.
For this purpose let us define
\begin{align}
&\hat L_k^1(\beta,\lambda)
 = \{-G(\beta-\beta_0) - \hat g_k^1(\beta_0) \}'\lambda
 - \frac{1}{2}\lambda'\Omega\lambda,\notag\\
&\hat L^2_{n,k}(\beta,\eta)
 = \{-G(\beta-\beta_0) - \hat g_{n,k}^2(\beta_0) \}'\eta
 - \frac{1}{2}\eta'\Omega\eta\notag
\end{align}
and
\[
\hat L_{n,k}(\beta,\lambda,\eta):=
k  \hat L^1_k(\beta,\lambda) + (n-k ) \hat L_{n,k}^2(\beta,\eta).\notag
\]
Furthermore, hereafter redefine
\begin{align}
&\tilde\beta_{n,k} := \arg\min_{\beta\in{\cal B}}\sup_{\lambda\in\R^m,\eta\in\R^m}
\hat L_{n,k}(\beta,\lambda,\eta),\notag\\
&\tilde\lambda_{n,k} := \arg\max_{\lambda\in\R^m}
\hat L_k^1(\tilde\beta,\lambda)\quad\text{ and }\quad
\tilde\eta_{n,k} := \arg\max_{\eta\in\R^m}
\hat L^2_{n,k}(\tilde\beta,\eta).\notag
\end{align}

\begin{lemma}\label{lem:A4}
Suppose that Assumptions \ref{ass:1}-\ref{ass:4} hold.
Then, under $H_0$,
$$\hat P_{n, k^*}(\hat\beta_{n,{k^*}},\hat\lambda_{n,{k^*}},\hat\eta_{n,{k^*}})
 = \hat L_{n, k^*}(\tilde\beta_{n,{k^*}},\tilde\lambda_{n,{k^*}},\tilde\eta_{n,{k^*}}) + o_p(1)$$
as $n \to \infty$.
\end{lemma}

\begin{proof}
It is sufficient to show that
\begin{enumerate}[label=(\roman*)]
\item $\hat P_{n,{k^*}}(\hat\beta_{n,{k^*}},\hat\lambda_{n,{k^*}},\hat\eta_{n,{k^*}})
- \hat L_{n,{k^*}}(\hat\beta_{n,{k^*}},\hat\lambda_{n,{k^*}},\hat\eta_{n,{k^*}})=o_p(1),$
\item $\hat L_{n,{k^*}}(\hat\beta_{n,{k^*}},\hat\lambda_{n,{k^*}},\hat\eta_{n,{k^*}})
- \hat L_{n,{k^*}}(\tilde\beta_{n,{k^*}},\hat\lambda_{n,{k^*}},\hat\eta_{n,{k^*}})=o_p(1),$
\item $\hat L_{n,{k^*}}(\tilde\beta_{n,{k^*}},\hat\lambda_{n,{k^*}},\hat\eta_{n,{k^*}})
- \hat L_{n,{k^*}}(\tilde\beta_{n,{k^*}},\tilde\lambda_{n,{k^*}},\tilde\eta_{n,{k^*}})=o_p(1).$
\end{enumerate}
For a proof of (i)  we note that a Taylor expansion leads to
\begin{align}
\hat{P}_k^1(\hat\beta_{n, k},\hat\lambda_{n, k})
 &=
-\hat\lambda_{n, k}' \hat{\hat g}^1_{n, k}
+\frac{1}{2}\hat\lambda'_{n, k}
\Big[ \frac{1}{k}\sum^{k}_{i=1}\rho_{i}^1(\ddot\lambda) a^*(X_{i-1})a^*(X_{i-1})'\Big]
\hat\lambda_{n, k},\notag
\end{align}
where $\ddot \lambda$ is on the line joining the points $\hat\lambda_{n, k}$ and $0_m$.
Observing the definition of $\hat L^1(\hat\beta_{n,k}, \hat\lambda_{n,k})$ this yields the estimate
\begin{eqnarray}
\left|\hat{P}_k^1(\hat\beta_{n, k},\hat\lambda_{n, k})
- \hat L^1(\hat\beta_{n, k},\hat\lambda_{n, k})\right|
 &\leq & \Big  | -\left( \hat{\hat g}^1_{n, k} - \hat{g}_k^1(\beta_0)
	- G(\hat\beta-\beta_0) \right)'\hat\lambda_{n, k} \Big | \label{eq:A6.1}\\  \nonumber
 & + & \Big|\frac{1}{2}\hat\lambda'_{n, k}
	\Big[\frac{1}{k}\sum^{k}_{i=1}\dot \rho^1_{i} a^*(X_{i-1})a^*(X_{i-1})' + \Omega
	\Big]\hat\lambda_{n, k}
	\Big|. \label{eq:A6.2}
\end{eqnarray}
Since $\hat\beta_{n,{k^*}}\plim\beta_0$ by Theorem \ref{thm:1},
we can take $\overline\beta_{n, k^*}=\hat\beta_{n,{k^*}}$ in Lemma \ref{lem:A2}, and obtain $\hat\lambda_{n,{k^*}}=O_p(n^{-1/2})$.
Then, recalling \eqref{eq:A7}, the first term in \eqref{eq:A6.1} (where $k$ is replaced by $k^*$) becomes
\begin{align*}
&\Big | -\left( \hat{\hat g}^1_{n, {k^*}} - \hat{g}_{k^*}^1(\beta_0)
	- G(\hat\beta_{n, k^*}-\beta_0) \right)'\hat\lambda_{n, k^*} \Big |\\
\leq&
\left\{
\left\|\hat{\hat g}^1_{n, {k^*}} - \hat{g}_{k^*}^1(\beta_0) - g(\hat\beta_{n, k^*})\right\|
+\left\| g(\hat\beta_{n, k^*}) - G(\hat\beta_{n, k^*} - \beta_0) \right\|
\right\}\left\| \hat\lambda_{n, k^*} \right\|\notag\\
=&
\Big \{
\left(1+\sqrt{n}\left\|\hat\beta_{n, k^*}-\beta_0\right\|\right)o_p(n^{-1/2})
+ O_p\Big(\left\| \hat\beta_{n, k^*} - \beta_0 \right\|^2\Big)
\Big \} O_p(n^{-1/2})\notag\\
=&
o_p(n^{-1}).
\end{align*}
Moreover,  the second term in \eqref{eq:A6.1} is of order $o_p({k^*}^{-1})$.
Hence, we get
$$
\big|\hat{P}_{k^*}^1(\hat\beta_{n, k^*},\hat\lambda_{n, k^*})
- \hat L^1(\hat\beta_{n, k}, \hat\lambda_{n, k^*})\big | = o_p({k^*}^{-1})$$
  and similarly
$$
\big |\hat P^2_{n, k^*}(\hat\beta_{n, k^*},\hat\eta_{n, k^*})
- \hat L^2(\hat\beta_{n, k^*},\hat\eta_{n, k^*})\big | = o_p((n-k^*)^{-1}).$$
Combining these estimates yields
\begin{align*}
\hat P_{n, k^*}(\hat\beta_{n, k^*},\hat\lambda_{n, k^*},\hat\eta_{n, k^*})
- \hat L_{n, k^*}(\hat\beta_{n, k^*},\hat\lambda_{n, k^*},\hat\eta_{n, k^*})=o_p(1),
\end{align*}
which is the statement (i).

For a proof of (ii)  we first show
$$ \big | \hat P_{n,{k^*}}(\tilde\beta_{n,{k^*}},\hat\lambda_{n,{k^*}},\hat\eta_{n,{k^*}})
-\hat L_{n,{k^*}}(\tilde\beta_{n,{k^*}},\hat\lambda_{n,{k^*}},\hat\eta_{n,{k^*}}) \big | = o_p(1).$$
Note that the function $\hat L_{n,k}(\beta,\lambda,\eta)$ is smooth in $\beta$, $\lambda$
and $\eta$.
Then, the first order conditions for an interior global maximum
\begin{align}
&0_p = \frac{\partial\hat L_{n,k}(\beta,\lambda,\eta)}{\partial\beta}
 =-G'\left\{k \lambda + (n-k )\eta \right\}
,\notag\\
&0_m = \frac{\partial\hat L_{n,k}(\beta,\lambda,\eta)}{\partial\lambda}
= -k  \left\{ G(\beta-\beta_0) +\hat g_k^1(\beta_0) + \Omega\lambda \right\}
\notag , \\
&0_m = \frac{\partial\hat L_{n,k}(\beta,\lambda,\eta)}{\partial\eta}
= -(n-k)\left\{ G(\beta-\beta_0) +\hat g_{n,k}^2(\beta_0) + \Omega\eta \right\}
\notag
\end{align}
 are satisfied for the point $(\beta',\lambda',\eta') = (\tilde\beta'_{n, k},\tilde\lambda'_{n, k},\tilde\eta'_{n, k})$.
These conditions can be rewritten in matrix form as
\begin{align}
\left(
\begin{array}{lll}
O_{p\times p}& G' & G'\\
G & k^{-1}\Omega & O_{m\times m}\\
G & O_{m\times m} & (n-k)^{-1}\Omega
\end{array}
\right)
\left(
\begin{array}{c}
\tilde\beta_{n,k}-\beta_0\\
k\tilde\lambda_{n,k}\\
(n-k)\tilde\eta_{n,k}
\end{array}
\right)
+
\left(
\begin{array}{c}
0_p\\
\hat g_k^1(\beta_0)\\
\hat g_{n,k}^2(\beta_0)
\end{array}
\right) = 0_{p+2m}.\label{eq:A8}
\end{align}
With the notations
\begin{eqnarray*}
 \Sigma &:=& (G'\Omega^{-1}G)^{-1}, \quad
 H:=\Omega^{-1}G\Sigma, \\
 P_k^1 &:=&\Omega^{-1}-\frac{k}{n}H\Sigma^{-1}H', \quad
 P_{n,k}^2:=\Omega^{-1}-\frac{n-k}{n}H\Sigma^{-1}H',
\end{eqnarray*}
the system \eqref{eq:A8} is equivalent to
\begin{align}\nonumber
& \left(
\begin{array}{c}
\tilde\beta_{n, k}-\beta_0\\
{k}\tilde\lambda_{n, k}\\
(n - k)\tilde\eta_{n, k}
\end{array}
\right) \\
& =
\left(
\begin{array}{lll}
 n^{-1}\Sigma & -{k}n^{-1}H' & -(n-{k})n^{-1}H' \\
-{k}n^{-1}H & -{k}P^1_k & {k}(n-{k})n^{-1}H\Sigma^{-1}H' \\
-(n-k)n^{-1}H & {k}(n-{k})n^{-1}H\Sigma^{-1}H' & -(n- k)P^2_{n, k}
\end{array}
\right)
\left(
\begin{array}{c}
0_p\\
\hat{g}_k^1(\beta_0)\\
\hat g_{n, k}^2(\beta_0)
\end{array}
\right)\notag\\
&= \left(
\begin{array}{c}
-H' \hat g(\beta_0) \\
-k\left\{\Omega^{-1} \hat{g}_k^1(\beta_0) - H\Sigma^{-1}H'\hat g(\beta_0)\right\}\\
-(n-k)\left\{\Omega^{-1} \hat g_{n, k}^2(\beta_0) - H\Sigma^{-1}H'\hat g(\beta_0)\right\}
\end{array}
\right). \label{eq:A9}
\end{align}
Consequently, $\tilde\beta_{n, k^*} - \beta_0$, $\tilde\lambda_{n, k^*}$ and $\tilde\eta_{n, k^*}$ are of order
$O_p(n^{-1/2})$, $O_p({k^*}^{-1/2})$ and $O_p((n-{k^*})^{-1/2})$, respectively.
Therefore, by the same arguments as given in the proof of (i), it follows that
\[
| \hat P_{n,{k^*}}(\tilde\beta_{n,{k^*}},\hat\lambda_{n,{k^*}},\hat\eta_{n,{k^*}})
-\hat L_{n,{k^*}}(\tilde\beta_{n,{k^*}},\hat\lambda_{n,{k^*}},\hat\eta_{n,{k^*}})| = o_p(1).
\]
This relationship and the fact that $(\hat\beta_{n,{k^*}}',\hat\lambda_{n,{k^*}}',\hat\eta_{n,{k^*}}')'$ and
$(\tilde\beta_{n,{k^*}}',\tilde\lambda_{n,{k^*}}',\tilde\eta_{n,{k^*}}')'$ are   the saddle points of the functions
$\hat P_{n,{k^*}}(\beta,\lambda,\eta)$ and $\hat L_{n,{k^*}}(\beta,\lambda,\eta)$, respectively,
imply that
\begin{align}
\hat L_{n,{k^*}}(\hat\beta_{n,{k^*}},\hat\lambda_{n,{k^*}},\hat\eta_{n,{k^*}})
 &= \hat P_{n,{k^*}}(\hat\beta_{n,{k^*}},\hat\lambda_{n,{k^*}},\hat\eta_{n,{k^*}}) + o_p(1)\notag\\
 &\leq \hat P_{n,{k^*}}(\tilde\beta_{n,{k^*}},\hat\lambda_{n,{k^*}},\hat\eta_{n,{k^*}}) + o_p(1)\notag\\
 &= \hat L_{n,{k^*}}(\tilde\beta_{n,{k^*}},\hat\lambda_{n,{k^*}},\hat\eta_{n,{k^*}}) + o_p(1).\label{eq:A10}
\end{align}
On the other hand,
\begin{align}
\hat L_{n,{k^*}}(\tilde\beta_{n,{k^*}},\hat\lambda_{n,{k^*}},\hat\eta_{n,{k^*}})
&\leq \hat L_{n,{k^*}}(\tilde\beta_{n,{k^*}},\tilde\lambda_{n,{k^*}},\tilde\eta_{n,{k^*}})\notag\\
&\leq \hat L_{n,{k^*}}(\hat\beta_{n,{k^*}},\tilde\lambda_{n,{k^*}},\tilde\eta_{n,{k^*}})\notag\\
&= \hat P_{n,{k^*}}(\hat\beta_{n,{k^*}},\tilde\lambda_{n,{k^*}},\tilde\eta_{n,{k^*}}) + o_p(1)\notag\\
&\leq \hat P_{n,{k^*}}(\hat\beta_{n,{k^*}},\hat\lambda_{n,{k^*}},\hat\eta_{n,{k^*}}) + o_p(1)\notag\\
&= \hat L_{n,{k^*}}(\hat\beta_{n,{k^*}},\hat\lambda_{n,{k^*}},\hat\eta_{n,{k^*}}) + o_p(1).\label{eq:A11}
\end{align}
Thus, \eqref{eq:A10} and \eqref{eq:A11} lead to
$$\hat L_{n,{k^*}}(\hat\beta_{n,{k^*}},\hat\lambda_{n,{k^*}},\hat\eta_{n,{k^*}})
-\hat L_{n,{k^*}}(\tilde\beta_{n,{k^*}},\hat\lambda_{n,{k^*}},\hat\eta_{n,{k^*}})=o_p(1).$$

Finally, we can prove (iii) by similar arguments  that
\begin{align*}
\hat L_{n,{k^*}}(\tilde\beta_{n,{k^*}},\tilde\lambda_{n,{k^*}},\tilde\eta_{n,{k^*}})
&\leq \hat L_{n,{k^*}}(\hat\beta_{n,{k^*}},\tilde\lambda_{n,{k^*}},\tilde\eta_{n,{k^*}})\notag\\
&= \hat P_{n,{k^*}}(\hat\beta_{n,{k^*}},\tilde\lambda_{n,{k^*}},\tilde\eta_{n,{k^*}}) + o_p(1)\notag\\
&\leq \hat P_{n,{k^*}}(\hat\beta_{n,{k^*}},\hat\lambda_{n,{k^*}},\hat\eta_{n,{k^*}}) + o_p(1)\notag\\
&\leq \hat P_{n,{k^*}}(\tilde\beta_{n,{k^*}},\hat\lambda_{n,{k^*}},\hat\eta_{n,{k^*}}) + o_p(1)\notag\\
&= \hat L_{n,{k^*}}(\tilde\beta_{n,{k^*}},\hat\lambda_{n,{k^*}},\hat\eta_{n,{k^*}}) + o_p(1)
\end{align*}
and
\begin{align*}
\hat L_{n,{k^*}}(\tilde\beta_{n,{k^*}},\hat\lambda_{n,{k^*}},\hat\eta_{n,{k^*}})\leq
\hat L_{n,{k^*}}(\tilde\beta_{n,{k^*}},\tilde\lambda_{n,{k^*}},\tilde\eta_{n,{k^*}}).
\end{align*}
Consequently, $\hat L_{n,{k^*}}(\tilde\beta_{n,{k^*}},\hat\lambda_{n,{k^*}},\hat\eta_{n,{k^*}})=
\hat L_{n,{k^*}}(\tilde\beta_{n,{k^*}},\tilde\lambda_{n,{k^*}},\tilde\eta_{n,{k^*}})+o_p(1)$, which implies (iii).\end{proof}

\begin{proof}\renewcommand{\qedsymbol}{}[Proof of Theorem \ref{thm:2}]
By \eqref{eq:l5.2}, \eqref{eq:l5.3}, Lemma \ref{lem:A4} and \eqref{eq:A9} it follows that
\begin{align}
\sup_{\beta\in{\cal B}}\{-l_{n, {k^*}}(\beta,\beta)\}
 &= \hat P_{n,{k^*}}(\hat\beta_{n,{k^*}},\hat\lambda_{n,{k^*}},\hat\eta_{n,{k^*}})\notag\\
 &=
\hat L_{n,{k^*}}(\tilde\beta_{n,{k^*}},\tilde\lambda_{n,{k^*}},\tilde\eta_{n,{k^*}}) + R_{n,{k^*}}\notag\\
 &=
\frac{{k^*}}{2}\tilde\lambda_{n,{k^*}}' \Omega \tilde\lambda_{n,{k^*}}
+ \frac{n-{k^*}}{2}\tilde\eta_{n,{k^*}}' \Omega \tilde\eta_{n,{k^*}} + R_{n,{k^*}} + o_p(1)\notag\\
 &=
\frac{{k^*}}{2}\hat g_{k^*}^1(\beta_0)' \Omega^{-1}\hat g_{k^*}^1(\beta_0)
+ \frac{n-{k^*}}{2}\hat g_{n,{k^*}}^2(\beta_0)' \Omega^{-1}\hat g_{n,{k^*}}^2(\beta_0) \notag\\
& \quad -
\frac{n}{2}\hat g(\beta_0)' H\Sigma^{-1}H'\hat g(\beta_0)
 + R_{n, {k^*}}+ o_p(1)\notag\\
 &= \hat M_{n, {k^*}} + R_{n, {k^*}}+ o_p(1),\label{eq:A12}
\end{align}
where
\begin{align*}
& \hat M_{n, k}
 = \frac{\big\| \hat W_n(k/n) - (k/n)\hat W_n(1) \big\|^2}{2\phi(k/n)}
+ \frac{\hat W_n(1)'Q\hat W_n(1)}{2},   \\
&\hat W_n(r) = \frac{1}{\sqrt{n}}\sum^{[r n]}_{t=1}\Omega^{-1/2}g(\mathscr{Y}^{p}_t, \beta_0),\notag\\
&R_{n, k} = \hat P_{n, k}(\hat\beta_{n, k},\hat\lambda_{n, k},\hat\eta_{n, k})-
\hat L_{n, k}(\tilde\beta_{n, k},\tilde\lambda_{n, k},\tilde\eta_{n, k}),
\end{align*}
$\phi(u) = u(1- u)$
and $[x]$ denotes the integer part of real number $x$.
As shown in Lemma \ref{lem:A4},
\[
\max_{k_{1n}\leq k^*\leq k_{2n}} |R_{n, k^*}|
=
\sup_{r_1 \leq r \leq r_2} |R_{n, rn}|
=
o_p(1).
\]

Second, from Assumption \ref{ass:4}
and Lemma 2.2 in \cite{phillips1987time},
it follows that
\begin{align*}
\left\{c'\hat W_{n}(r):r\in[0,1]\right\} \dlim \left\{c' B(r):r\in[0,1]\right\},
\end{align*}
for any vector $c\in\R^m$,
where $\{B(r):r\in[0,1]\}$ is an $m$-dimensional standard Brownian motion.
Hence, the Cram{\'e}r-Wold device and the continuous mapping theorem lead to
\begin{align*}
\tilde{T}_n
&= 2 \max_{k_{1n}\leq k \leq k_{2n}}\Big\{h\Bigl(\frac{k}{n} \Bigr)
\hat M_{n, k}\Big\} \\
&~~~~~~~~~
=\sup_{k_{1n}/n\leq r \leq k_{2n}/n}
\Big\{
	\frac{h(k/n)}{\phi(k/n)}\big\| \hat W_n(r) - ([rn]/n)\hat W_n(1) \big\|^2  +h([rn]/n)
	\hat W_n(1)'Q\hat W_n(1)
	\Big\}\notag\\
&~~~~~~~~~\dlim
\sup_{r_1\leq r\leq r_2}
\Big\{\frac{h(r)}{\phi(r)}\left\| B(r)- r  B(1)\right\|^2+h(r) B(1)'Q  B(1)\Big\}.\notag
\end{align*}
\end{proof}

\subsection{Proof of Theorem \ref{thm:3}}
\begin{proof}\renewcommand{\qedsymbol}{}
Without loss of generality, suppose that $\theta_2 \not = \beta_0$.
This implies that there exist a neighborhood $U(\beta_0)$ of $\beta_0$
and a neighborhood $U(\theta_2)$ of $\theta_2$ such that
$$U(\beta_0) \cap U(\theta_2) = \emptyset.$$
Under the alternative  it follows that $\hat{\beta}_{n, k^*} \not \in U(\beta_0)$
or $\hat{\beta}_{n, k^*} \not \in U(\theta_2)$.
Note that
$\E[g({\cal Y}^p_t,\theta_2)]\neq0$ for $1\leq t\leq k^*$ and $\E[g({\cal Y}^p_t,\beta_0)]\neq0$ for $k^*+1\leq t\leq n$.
 From a uniform law of large numbers, $\hat{g}^1_{k^*}(\hat{\beta}_{n, k^*})$
or $\hat{g}^2_{n, k^*}(\hat{\beta}_{n, k^*})$ is outside a neighborhood of 0
for any sufficiently large   $n$.

Now, if we consider $\hat{g}_k^1(\hat{\beta}_{n, k^*})$ instead of $\hat{g}_k^1(\beta_0)$
and $\hat{g}_{n, k}^2(\hat{\beta}_{n, k^*})$ instead of $\hat{g}_{n, k}^2(\beta_0)$ in \eqref{eq:A8},
we find, as in \eqref{eq:A12}, that $\sup_{\beta \in \B} \{-l_{n, k^*}(\beta, \beta)\}$ can be approximated by
\begin{multline*}
\frac{{k^*}}{2}\hat g_{k^*}^1(\hat{\beta}_{n, k^*})' \Omega^{-1}\hat g_{k^*}^1(\hat{\beta}_{n, k^*})
+ \frac{n-{k^*}}{2}\hat g_{n,{k^*}}^2(\hat{\beta}_{n, k^*})' \Omega^{-1}\hat g_{n,{k^*}}^2(\hat{\beta}_{n, k^*}) \\
 -
\frac{n}{2}\hat g(\hat{\beta}_{n, k^*})' H\Sigma^{-1}H'\hat g(\hat{\beta}_{n, k^*})
 + R_{n, {k^*}}+ o_p(1).
\end{multline*}
This time, however, we have
\[
\frac{{k^*}}{2}\hat g_{k^*}^1(\hat{\beta}_{n, k^*})' \Omega^{-1}\hat g_{k^*}^1(\hat{\beta}_{n, k^*})
+ \frac{n-{k^*}}{2}\hat g_{n,{k^*}}^2(\hat{\beta}_{n, k^*})' \Omega^{-1}\hat g_{n,{k^*}}^2(\hat{\beta}_{n, k^*})
\to \infty,
\]
since $\hat g_{k^*}^1(\hat{\beta}_{n, k^*})' \Omega^{-1}\hat g_{k^*}^1(\hat{\beta}_{n, k^*})
+ \hat g_{n,{k^*}}^2(\hat{\beta}_{n, k^*})' \Omega^{-1}\hat g_{n,{k^*}}^2(\hat{\beta}_{n, k^*}) > 0$
for any sufficiently large   $n$.
This completes the proof of Theorem \ref{thm:3}.
\end{proof}

\bigskip
\bigskip
\noindent
{\bf Acknowledgements.}
The authors would like to thank Martina Stein who typed this manuscript with considerable technical expertise.
The work of  authors was supported by
JSPS Grant-in-Aid for Young Scientists (B) (16K16022), Waseda University Grant for Special Research Projects (2016S-063) and
the Deutsche  Forschungsgemeinschaft (SFB 823: Statistik nichtlinearer dynamischer Prozesse, Teilprojekt A1 and C1).

\bibliographystyle{chicago}
\setlength{\bibsep}{1pt}
\begin{small}
\bibliography{ref}
\end{small}
\end{document}